\documentclass[11pt]{article}
\usepackage[a4paper, margin=0.9in]{geometry}
\usepackage{float}
\usepackage{amsmath}
\usepackage{amssymb}
\usepackage{amsfonts}
\usepackage{amsthm}
\usepackage{amsopn}
\usepackage{graphicx}
\usepackage{epstopdf}
\usepackage{setspace}
\usepackage{xcolor}
\usepackage{caption}
\usepackage{subcaption}
\usepackage{comment}
\usepackage{booktabs}
\usepackage{multirow}
\usepackage{stmaryrd}
\usepackage{algorithm}
\usepackage[noEnd=true,indLines=true,italicComments=false]{algpseudocodex}
\usepackage[american]{babel}
\usepackage{csquotes}
\usepackage{chngcntr}
\usepackage{hyperref}
\usepackage{cleveref}
\hypersetup{
  colorlinks=true,
  linkcolor=blue,
  citecolor=blue,
  urlcolor=blue,
  pdftitle={A Semismooth Newton Augmented Lagrangian Method for Sparse Spectral Risk Optimization},
  pdfauthor={Rufeng Xiao, Rujun Jiang, Xudong Li, and Defeng Sun},
  pdfkeywords={spectral risk measure, semismooth Newton method, pool adjacent violators algorithm, augmented Lagrangian method}
}

\DeclareGraphicsExtensions{.pdf,.png,.jpg,.eps}

\counterwithout{figure}{section}
\counterwithout{table}{section}

\let\STATE\State
\let\IF\If
\let\ENDIF\EndIf
\let\WHILE\While
\let\ENDWHILE\EndWhile
\let\FOR\For
\let\ENDFOR\EndFor

\newtheorem{theorem}{Theorem}[section]
\newtheorem{proposition}[theorem]{Proposition}
\newtheorem{lemma}[theorem]{Lemma}
\newtheorem{assumption}{Assumption}
\newtheorem{definition}[theorem]{Definition}

\newtheorem{remark}[theorem]{Remark}

\newenvironment{keywords}{\par\smallskip\noindent\textbf{Keywords.}\ }{\par\smallskip}
\newenvironment{MSCcodes}{\par\smallskip\noindent\textbf{MSC codes.}\ }{\par\smallskip}

\crefname{assumption}{Assumption}{Assumptions}
\crefname{claim}{Claim}{Claims}
\crefname{fact}{Fact}{Facts}

\newcommand{\dist}{\operatorname{dist}}
\newcommand{\bs}{\boldsymbol}

\newcommand{\bb}{\mathbb}

\newcommand{\prox}{\operatorname{prox}}

\def\R{{\mathbb{R}}}

\def\ba{{\boldsymbol a}}
\def\bb{{\boldsymbol b}}
\def\bc{{\boldsymbol c}}
\def\bd{{\boldsymbol d}}

\def\br{{\boldsymbol r}}

\def\bu{{\boldsymbol u}}
\def\bv{{\boldsymbol v}}
\def\bw{{\boldsymbol w}}
\def\bx{{\boldsymbol x}}
\def\by{{\boldsymbol y}}
\def\bz{{\boldsymbol z}}

\def\blambda{{\boldsymbol \lambda}}

\title{A Semismooth Newton Augmented Lagrangian Method for Sparse Spectral Risk Optimization}
\author{Rufeng Xiao\thanks{School of Data Science, Fudan University, Shanghai, China. Email: \href{mailto:rfxiao24@m.fudan.edu.cn}{rfxiao24@m.fudan.edu.cn}.}
\and Rujun Jiang\thanks{School of Data Science, Fudan University, Shanghai, China. Email: \href{mailto:rjjiang@fudan.edu.cn}{rjjiang@fudan.edu.cn}.}
\and Xudong Li\thanks{School of Data Science, Fudan University, Shanghai, China. Email: \href{mailto:lixudong@fudan.edu.cn}{lixudong@fudan.edu.cn}.}
\and Defeng Sun\thanks{Department of Applied Mathematics, The Hong Kong Polytechnic University, Hung Hom, Kowloon, Hong Kong. Email: \href{mailto:defeng.sun@polyu.edu.hk}{defeng.sun@polyu.edu.hk}.}}
\date{}
\begin{document}

\maketitle

\begin{abstract}
Empirical risk minimization is a standard and effective paradigm for learning predictive models by minimizing average loss. In high-stakes decision-making, however, an average-loss criterion may underrepresent rare but severe losses. Spectral risk measures (SRMs) provide a principled framework by incorporating weighted order statistics of losses, but the induced nonsmoothness and nonseparability from sorting make the resulting optimization problems challenging. We propose a relative inexact proximal augmented Lagrangian method with a semismooth Newton subproblem solver for solving SRM-based optimization problems. Exploiting a dual reformulation and properties of the Moreau envelope, we reduce the subproblems to structured dual-variable formulations, significantly simplifying computation. We provide explicit generalized Jacobian characterizations and tailor the pool adjacent violators algorithm for their efficient evaluation. Numerical results on synthetic and real-data instances show that the proposed method attains lower running times than the tested ADMM baseline while producing comparable stationarity residuals and sparse solutions.
\end{abstract}



\section{Introduction}
Empirical risk minimization (ERM) optimizes average predictive performance but often struggles with robustness in high-stakes environments, where extreme, infrequent losses dictate system reliability. To incorporate risk sensitivity and fairness, alternative formulations penalizing tail behavior have gained traction \cite{acerbi2002portfolio,acerbi2002spectral,chow2015risk, dwork2012fairness}.

Spectral risk measures (SRMs) provide a unifying framework for this purpose \cite{acerbi2002spectral, maurer2021robust}. An SRM aggregates individual losses using nonnegative weights on their order statistics, enabling a controlled emphasis on extreme outcomes.
By calibrating these weights, SRMs bridge classical ERM, worst-case risk, and conditional value-at-risk (CVaR)~\cite{rockafellar2000optimization}, while subsuming criteria like extremiles \cite{daouia2019extremiles}. Specifically, our numerical experiments focus on three representative SRMs with weight coefficients defined respectively by
\begin{align}
    \text{$\nu$-Superquantile:} \quad & \sigma_i = \frac{1}{n\nu} \mathbb{I}_{\{i > k\}} + \left(1 - \frac{ \lfloor n\nu \rfloor}{n\nu}\right) \mathbb{I}_{\{i = k\}}, \quad && \text{for } \nu \in (0,1), \label{eq:srm_cvar} \\
    \text{$\varrho$-ESRM:} \quad & \sigma_i = \frac{e^{-\varrho}\bigl(e^{\varrho i/n} - e^{\varrho (i-1)/n}\bigr)}{1 - e^{-\varrho}}, \quad && \text{for } \varrho > 0, \label{eq:srm_esrm} \\
    \text{$r$-Extremile:} \quad & \sigma_i = \left( \frac{i}{n} \right)^r - \left( \frac{i-1}{n} \right)^r, \quad && \text{for } r \geq 1, \label{eq:srm_extremile}
\end{align}
where $k = \lceil n(1-\nu) \rceil$, $\mathbb{I}_{\{\cdot\}}$ is the indicator function, and spectral weights satisfy $0\le
\sigma_1\le\cdots\le\sigma_n$ with $\sum_{i=1}^n\sigma_i=1$.

Motivated by this unifying perspective, we formulate the sparsity-promoting spectral risk optimization problem as
\begin{equation}
\label{eq:primal-problem}
\min_{\boldsymbol{w} \in \mathbb{R}^{d}}
\mathcal{F}(\bw) \triangleq  F(\boldsymbol{w}) + \lambda \|\boldsymbol{w}\|_1,
\qquad
F(\boldsymbol{w}) \triangleq  \sum_{i=1}^{n} \sigma_i L_{[i]}(\boldsymbol{w}),
\end{equation}
where $\lambda > 0$ is a regularization parameter, $\sigma_i \in [0,1]$ are the spectral weights, and $[i]$ denotes the index of the $i$-th smallest component of the loss vector $L(\boldsymbol{w})$. To characterize the loss structure, we express the loss vector as the composite mapping $L(\boldsymbol{w}) = \ell(D\boldsymbol{w} + \boldsymbol{c})$, where $D \in \mathbb{R}^{n \times d}$ and $\boldsymbol{c} \in \mathbb{R}^n$ are the given data parameters. Here, the vector-valued mapping
\begin{equation}
\label{eq:individual-loss-def}
\ell(\boldsymbol{z})
\triangleq
\bigl[l(z_1), \ldots, l(z_n)\bigr]^\top,
\end{equation}
applies the individual loss function $l:\mathbb{R}\to\mathbb{R}_+$ componentwise. Crucially, the affine mapping $D\boldsymbol{w}+\boldsymbol{c}$ characterizes the vector of prediction scores (or margins), placing regression and classification within the same composite optimization model. From an application standpoint, this objective systematically upweights high-loss samples, providing a robust mechanism for safety- and fairness-sensitive learning \cite{dwork2012fairness}.

Solving \eqref{eq:primal-problem} is algorithmically challenging because the ranking operation induces nonsmoothness and nonseparability. Subgradient methods have theoretical convergence guarantees but are often too slow for large-scale instances. Existing deterministic approaches mainly exploit tractable reformulations in special cases, especially CVaR \cite{rockafellar2000optimization}. For example, \cite{wu2023convex} developed an augmented Lagrangian method (ALM) with a semismooth Newton (\textnormal{\textsc{Ssn}}) solver for absolute-value losses by reducing the problem to Ky-Fan $k$-norm minimization. For broader spectral risk objectives, ADMM-type methods have also been studied, including an ADMM algorithm for cumulative prospect theory that is adaptable to spectral objectives \cite{cui2025decision} and a generalized ADMM framework based on a modified Pool Adjacent Violators Algorithm (PAVA) \cite{xiao2023unified,best2000minimizing}. On the stochastic side, variance-reduced smoothing methods typically solve only smoothed approximations of \eqref{eq:primal-problem} \cite{mehta2023stochastic}. The recent SOREL method provides the first stochastic algorithm with exact convergence guarantees, but it requires a strongly convex regularizer \cite{ge2025sorel}.

To efficiently solve the sparsity-promoting problem \eqref{eq:primal-problem}, we propose a relative inexact proximal augmented Lagrangian method (\textnormal{\textsc{ripALM}}) equipped with an \textnormal{\textsc{Ssn}} solver, referred to as the \textnormal{\textsc{ripALM}}-\textnormal{\textsc{Ssn}} method.
Specifically, we construct a structured augmented Lagrangian based on the dual formulation. By using the Moreau envelope representation, we formulate the augmented Lagrangian subproblems in the dual variables only, while isolating the nonsmooth spectral-risk structure in the associated proximal mapping.
Furthermore, to efficiently evaluate the proximal operator associated with the spectral risk term, we develop a tailored, structure-exploiting implementation of PAVA. This refinement reduces the number of one-dimensional subproblem evaluations compared to classical schemes \cite{xiao2023unified,cui2025decision}.
Unlike settings where the proximal mapping admits a closed-form formula \cite{wu2023convex}, the mapping considered here is implicitly defined through PAVA. We derive structural and perturbation estimates that characterize generalized Jacobians of this mapping for the hinge loss and for a broad class of SC$^1$ loss functions. These characterizations supply the derivative information used in our \textnormal{\textsc{Ssn}} method.
Finally, numerical experiments on synthetic and real-data instances show lower running times than the tested ADMM baseline while producing sparse solutions and comparable stationarity residuals.

\textbf{Notation.} We use bold lowercase letters for vectors and regular letters for scalars. Let $\bs{1}_n, \bs{0}_n \in \mathbb{R}^n$ denote the all-ones and all-zeros vectors (subscripts omitted when clear), $I_n$ the $n$-dimensional identity matrix, and $\bs{e}_i$ the $i$-th standard basis vector. The standard inner product and Euclidean norm are denoted by $\langle\cdot,\cdot\rangle$ and $\|\cdot\|$, respectively. For a proper closed convex function $\psi: \R^q \to (-\infty, +\infty]$, its subdifferential at $\bx \in \operatorname{dom}\psi$ is $\partial \psi(\bx) \triangleq  \{\bd \in \R^q \mid \psi(\by) \geq \psi(\bx) + \langle\bd, \by-\bx\rangle, \forall \by\}$, and its conjugate is $\psi^*(\by) \triangleq  \sup_{\bx} \{\langle \by, \bx\rangle - \psi(\bx)\}$. For a vector $\bx$, $x_{[i]}$ denotes its $i$-th smallest element, and $\vec{\bx}\triangleq [x_{[1]},\dots,x_{[n]}]^\top$ denotes its version sorted in nondecreasing order. For integers \(i_1\) and \(i_2\), let \(\llbracket i_1,i_2\rrbracket\triangleq\{i_1,i_1+1,\ldots,i_2\}\) if \(i_1\le i_2\), and let \(\llbracket i_1,i_2\rrbracket\triangleq\emptyset\) otherwise. The notation \(\bx_{\llbracket i_1,i_2\rrbracket}\triangleq [x_{i_1},\dots,x_{i_2}]^\top\) extracts the corresponding subvector when \(i_1\le i_2\), and denotes the empty vector otherwise. The support of a nonnegative vector $\bd$ is $\operatorname{supp}(\bd) \triangleq  \{i \mid d_i > 0\}$. Finally, $\operatorname{BlkD}(A_1,\dots,A_m)$ forms a block diagonal matrix from given blocks, and $\nabla\ell(\bz) \triangleq  [l'(z_1),\dots,l'(z_n)]^\top$ represents the componentwise derivative.

\section{\texorpdfstring{Preliminaries and the \textnormal{\textsc{ripALM}} Scheme}{Preliminaries and the ripALM Scheme}}
\label{sec:ripalm-framework}
This section develops the dual augmented Lagrangian formulation underlying the proposed \textnormal{\textsc{ripALM}} scheme. After fixing the Moreau-envelope and proximal notation, we derive the dual reformulation of problem~\eqref{eq:primal-problem}, reduce the augmented-Lagrangian subproblem, and state the resulting relative inexact proximal augmented Lagrangian method.

We begin with the Moreau-envelope and proximal notation used throughout.
Let \(\psi:\mathbb{R}^n\to(-\infty,+\infty]\) be proper, closed, and convex, and let \(\rho>0\). We define its Moreau envelope and, under the scaled-function convention, its proximal mapping by

\[
M_{\psi}^{\rho}(\bx)
\triangleq \min_{\by\in\mathbb{R}^n}\left\{\psi(\by)+\frac{\rho}{2}\|\bx-\by\|_2^2\right\},
\;
\prox_{\rho\psi}(\bx)\triangleq
\arg\min_{\by\in\mathbb{R}^n}\left\{\rho\psi(\by)+\frac{1}{2}\|\bx-\by\|_2^2\right\}.
\]

This convention is used below for \(\prox_{\rho f}\) and \(\prox_{\rho g}\).
The following standard Moreau identities, stated for the same parameterization, will be used repeatedly~\cite[Theorems 6.45, 6.60, 6.67]{beck2017first}.

\begin{lemma}
\label{lem:moreau-lemma}
For any $\bx\in\mathbb{R}^n$ and $\rho>0$, we have
\[
\bx = \prox_{\rho\psi}(\bx)+\rho\prox_{\psi^*/\rho}(\bx/\rho),
\quad \text{and} \quad
M_{\psi}^{1/\rho}(\bx)+M_{\psi^*}^{\rho}(\bx/\rho)=\frac{1}{2\rho}\|\bx\|_2^2.
\]
Moreover, \(M_\psi^{1/\rho}\) is continuously differentiable on \(\mathbb{R}^n\) with gradient
\(\nabla M_\psi^{1/\rho}(\bx)=\rho^{-1}\bigl(\bx-\prox_{\rho\psi}(\bx)\bigr)\).
\end{lemma}

Throughout the rest of the paper, we let $g(\bw) \triangleq  \lambda\|\bw\|_1$ for simplicity.
By introducing an auxiliary variable $\bz = D\bw + \bc$, we rewrite \eqref{eq:primal-problem} as
\begin{equation}
\label{problem:constrained-problem}
\min_{\bw \in \R^d, \bz \in \R^n} \ f(\bz) + g(\bw) \qquad \mathrm{s.t.} \quad D\bw - \bz + \bc = \bs{0},
\end{equation}
where \(f(\bz)\triangleq \sum_{i=1}^n\sigma_i(\ell(\bz))_{[i]}\).
The corresponding Lagrangian is $\mathcal{L}(\bw,\bz,\bu) = f(\bz) + g(\bw) + \langle \bu, D\bw - \bz + \bc \rangle$, yielding the dual problem
\begin{equation}
\label{problem:pure-dual-problem}
\min_{\bu \in \R^n} \ f^*(\bu) + g^*(-D^\top \bu) - \langle \bu, \bc \rangle.
\end{equation}

To proceed, we first formalize the basic assumption on the individual loss.
\begin{assumption}\label{assp:l-nondecreasing}
The individual loss function $l: \R \to \R^+$ appearing in \eqref{eq:individual-loss-def} is convex and nondecreasing.
\end{assumption}
Assumption~\ref{assp:l-nondecreasing} covers a wide range of standard loss functions in machine learning, including the logistic, hinge, exponential, and the smoothed hinge loss defined as $l(z) = 0$ for $z \le -1$, $\frac{1}{2}(1+z)^2$ for $z \in (-1, 0]$, and $\frac{1}{2}+z$ for $z > 0$. Note that under Assumption~\ref{assp:l-nondecreasing}, since $l$ is nondecreasing, sorting $\ell(\bz)$ is equivalent to sorting $\bz$, which means that $f(\bz)$ has the form $\sum_{i=1}^n\sigma_i l(z_{[i]})$.

Under Assumption~\ref{assp:l-nondecreasing}, the functions $f$ and $g$ are proper, closed, and convex with full domains. Hence the relative-interior condition in \cite[Corollary~31.2.1]{rockafellar1970} holds, which yields zero duality gap and dual attainment.
Moreover, since $f\ge0$ and $\lambda>0$, we have
$\mathcal{F}(\bw)\ge \lambda\|\bw\|_1$, so $\mathcal{F}$ is level-bounded and hence attains its finite optimum \cite[Theorem~1.9]{rockafellar2009variational}.

By introducing auxiliary variables $\bs{\xi}$ and $\bs{\zeta}$, we obtain the following reformulation of \eqref{problem:pure-dual-problem},
\begin{equation}
\label{eq:dual-problem}
\begin{aligned}
    \min_{\bu,\bs{\xi},\bs{\zeta}} & \quad  f^*(\bs{\zeta}) + g^*(\bs{\xi}) - \left<\bu,\bc\right>  \qquad \mathrm{s.t.}  -D^\top \bu = \bs{\xi}, \bu = \bs{\zeta}.
\end{aligned}
\end{equation}
The augmented Lagrangian minimization with penalty parameter $\rho > 0$ can be evaluated by first minimizing over $\bs{\xi}$ and $\bs{\zeta}$. This operation explicitly invokes the definition of the Moreau envelope, yielding
\begin{align*}
&\min_{\bs{\xi}, \bs{\zeta}, \bu} L_\rho(\bs{\xi}, \bs{\zeta}, \bu; \bw, \bz) \\
= &~\min_{\bu} \left\{ -\langle \bu, \bc \rangle + M_{g^*}^\rho\left(\frac{\bw}{\rho} - D^\top \bu\right) + M_{f^*}^\rho\left(\bu + \frac{\bz}{\rho}\right) - \frac{\|\bw\|^2 + \|\bz\|^2}{2\rho} \right\}.
\end{align*}
Applying the Moreau identity (Lemma~\ref{lem:moreau-lemma}) to express $M_{g^*}^\rho$ and $M_{f^*}^\rho$ in terms of the primal functions $g$ and $f$, the joint minimization strictly reduces to the following unconstrained $\bu$-subproblem.
\begin{equation}
\label{eq:alm-subproblem}
\begin{aligned}
\min_{\bu}\ \bar{L}_\rho(\bu;\bw,\bz) \triangleq  &-\left\langle \bu, \bc \right\rangle - M_f^{1/\rho}(\rho\bu + \bz) - M_g^{1/\rho}(\bw - \rho D^\top \bu)\\
& + \frac{\|\bw - \rho D^\top \bu\|^2 + \|\rho\bu + \bz\|^2 - \|\bw\|^2 - \|\bz\|^2}{2\rho}.
\end{aligned}
\end{equation}

We apply \textnormal{\textsc{ripALM}}~\cite{zhu2025ripalm} to \eqref{eq:dual-problem}, using the reduced subproblem~\eqref{eq:alm-subproblem}. The resulting scheme, summarized in Algorithm~\ref{alg:rALM}, solves the original dual problem~\eqref{problem:pure-dual-problem}. Its global and local convergence guarantees rely on the standard KKT error bound condition stated below \cite{li2020asymptotically, yang2024corrected}.

\begin{algorithm}[ht!]
\caption{\textnormal{\textsc{ripALM}} for solving reformulation~\eqref{eq:dual-problem}}
\label{alg:rALM}
\begin{algorithmic}[1]

\Statex \textbf{Input:} $\tau \in [0, 1),\bu_0,\bz_0,\bx_0\in \R^n, D\in \R^{n \times d}, \bw_0 \in \R^d$, a positive nonincreasing sequence $\{\beta_k\}$ such that $\sup_{k\geq0} \{\beta_k\} < \infty$, and a positive nondecreasing sequence $\{\rho_k\}$
    \STATE $k=0$
    \WHILE{a termination criterion is not met}
        \STATE Inexactly solve the $\bu$-subproblem
        \begin{equation}
        \label{problem:inexact-u-subproblem}
        \min_{\bu}\left\{
        \varphi_k(\bu)
        \triangleq  \bar{L}_{\rho_k}(\bu;\bw_k,\bz_k)
        + \frac{\beta_k}{2\rho_k}\|\bu-\bu_k\|_2^2
        \right\},
        \end{equation}
        to obtain $\bu_{k+1}$ such that
        \[
        \begin{aligned}
            &2|\left<\bx_k - \bu_{k+1}, \rho_k \nabla \varphi_{k} (\bu_{k+1}) \right>| + \|\rho_k \nabla\varphi_{k} (\bu_{k+1}) \|^2 \\
            \leq & \tau \left(\|\prox_{\rho_k f} (\rho_k \bu_{k+1} + \bz_k) - \bz_k\|^2 + \|\prox_{\rho_k g} (\bw_k - \rho_k D^\top \bu_{k+1}) - \bw_k\|^2 \right.\\
            &\left. + \beta_k \|\bu_{k+1}-\bu_{k}\|^2 \right)
        \end{aligned}
        \]
        \STATE Update $\bw_{k+1}=\operatorname{prox}_{\rho_k g}\left(\bw_{k}-\rho_k D^\top\bu_{k+1}\right)$, $\bz_{k+1}= \operatorname{prox}_{\rho_k f}\left(\bz_k + \rho_k \bu_{k+1}\right)$, and $\bx_{k+1}=\bx_{k}-\rho_k \nabla \varphi_k (\bu_{k+1})$
        \STATE $k = k + 1$
    \ENDWHILE
\end{algorithmic}
\end{algorithm}

\begin{assumption}\label{assp:error-bound}
There exist constants $\kappa, r > 0$ such that for any $(\bw,\bz,\bu)$ satisfying
\[
\dist\big((\bw,\bz,\bu), (\partial \mathcal{L})^{-1}(\bs{0})\big) \le r,
\]
it holds that
\[
\dist\big((\bw,\bz,\bu), (\partial \mathcal{L})^{-1}(\bs{0})\big) \le \kappa \dist\big(\bs{0}, \partial \mathcal{L}(\bw,\bz,\bu)\big).
\]
\end{assumption}
\noindent Under Assumptions~\ref{assp:l-nondecreasing} and \ref{assp:error-bound}, Algorithm~\ref{alg:rALM} globally converges to an optimal primal-dual solution pair and exhibits a local linear convergence rate. We refer the reader to \cite[Theorems 3.1 and 3.2]{zhu2025ripalm} for precise statements regarding the parameter thresholds and asymptotic convergence factors. Notably, if the individual loss is the hinge or smoothed hinge loss, then \(f\) is piecewise linear--quadratic. Indeed, on each order cone \(z_{\pi_1}\le\cdots\le z_{\pi_n}\), the function \(f\) reduces to \(\sum_{i=1}^n\sigma_i l(z_{\pi_i})\). For the smoothed hinge loss, further partitioning by the finitely many scalar breakpoint regions of \(l\) gives a finite polyhedral partition on each element of which \(f\) is quadratic, while for the hinge loss it is piecewise linear. Moreover, \(g(\bw)=\lambda\|\bw\|_1\) is piecewise linear and hence piecewise linear--quadratic. Therefore, the corresponding Lagrangian of problem~\eqref{problem:constrained-problem} is piecewise linear--quadratic, and the error bound condition in Assumption~\ref{assp:error-bound} is inherently satisfied in this case~\cite{robinson1981some}.

To set up the \textnormal{\textsc{Ssn}} method for the inner \(\bu\)-subproblem~\eqref{problem:inexact-u-subproblem}, we first record the gradient of its objective. By Lemma~\ref{lem:moreau-lemma}, \(\varphi_k\) is differentiable and
\begin{equation}
\label{eq:nabla-phik}
\begin{aligned}&\nabla \varphi_k (\bu)\\
 =& -\left(
\rho_k \bu + \bz_k - \prox_{\rho_k f}(\rho_k \bu + \bz_k)
\right) + D\left(\rho_k D^\top\bu - \bw_k\right) - \bc\\
&\quad + D\left(\bw_k-\rho_k D^\top\bu-\prox_{\rho_k g}(\bw_k -\rho_k D^\top \bu)\right)
+ \rho_k \bu + \bz_k + \tfrac{\beta_k}{\rho_k}(\bu-\bu_k)\\
=&  \prox_{\rho_k f}(\rho_k \bu + \bz_k) - D\prox_{\rho_k g}(\bw_k -\rho_k D^\top \bu)+ \tfrac{\beta_k}{\rho_k}(\bu-\bu_k) - \bc.
\end{aligned}
\end{equation}
The representation \eqref{eq:nabla-phik} identifies the computational path followed below. Section~\ref{sec:prox-pava} develops an efficient PAVA-based evaluation of \(\prox_{\rho f}\). Section~\ref{sec:hs-jacobian} characterizes generalized derivatives of this mapping. Section~\ref{sec:ssn-linear-algebra} then uses these ingredients to construct an \textnormal{\textsc{Ssn}} solver and compute its Newton directions efficiently.

\section{Proximal Mapping and a Structure-Exploiting PAVA}
\label{sec:prox-pava}
This section focuses on computing the SRM proximal mapping \(\operatorname{prox}_{\rho f}\) appearing in \eqref{eq:nabla-phik}. We first reduce its evaluation to an ordered convex subproblem and then solve this subproblem by a structure-exploiting PAVA.

\subsection{Reduction to an Ordered Proximal Subproblem}
We first derive a computable representation of $\operatorname{prox}_{\rho f}(\cdot)$.
For a vector $\bb$, denote the permutation matrix set by
\begin{equation}
\label{eq:perm-mat}
\mathcal{P}(\bb) = \{P \in \R^{n\times n}|(P\bb)_i = b_{[i]}, i=1,\dots,n\},
\end{equation}
where $b_{[1]}\leq \dots \leq b_{[n]}$.
Thus, for any $P \in \mathcal{P}(\bb)$, we have $P\bb=\vec{\bb}$ and $P^{-1}\vec{\bb} = P^\top\vec{\bb} =\bb$.
In this subsection, \(\by\) denotes the ordered data vector in the PAVA subproblem. When evaluating \(\prox_{\rho f}(\bb)\), we take \(\by=\vec{\bb}\).
We consider the following problem
\begin{equation}
\label{problem:pav-problem}
\min_\bz\quad \sum_{i=1}^n \left(\rho\sigma_i l(z_i)+\frac{1}{2}( z_i-y_i )^2\right) 
\qquad \mathrm{s.t.} \quad  z_1\leq z_2\leq \cdots \leq z_n.
\end{equation}

The next lemma shows that, for any $\bb$, evaluating $\prox_{\rho f}(\bb)$ reduces to solving problem~\eqref{problem:pav-problem} with $\by=\vec{\bb}$.
\begin{lemma}
\label{lem:prox-permutation}
    Suppose Assumption~\ref{assp:l-nondecreasing} holds. Given a vector $\by$, let $\bz^*(\by)$ denote the unique optimal solution to \eqref{problem:pav-problem}. Then, for any $P \in \mathcal{P}(\bb)$ defined in \eqref{eq:perm-mat}, we have $\prox_{\rho f}(\bb) = P^\top \bz^*(P\bb) = P^\top \bz^*(\vec{\bb})$.
\end{lemma}

\begin{proof}
Let $\bar{\bz}=\prox_{\rho f}(\bb)$, i.e., the unique minimizer of $\sum_{i=1}^n \rho \sigma_i l(z_{[i]})+\frac12\|\bz-\bb\|^2$. By the rearrangement argument in \cite[Lemma 3]{cui2018portfolio} and \cite[Section 4]{cui2025decision}, if $b_i>b_j$ then $\bar z_i\ge \bar z_j$; otherwise, swapping $\bar z_i$ and $\bar z_j$ leaves $\sum_{i=1}^n \rho \sigma_i l(z_{[i]})$ unchanged and strictly decreases the quadratic term. If $b_i=b_j$, the same swap leaves the whole objective unchanged, and uniqueness implies $\bar z_i=\bar z_j$. Hence $\bar{\bz}$ has the same sorting pattern as $\bb$, so $P\bar{\bz}$ is nondecreasing for every $P\in\mathcal{P}(\bb)$.

Now fix any $P\in\mathcal{P}(\bb)$ and set $\hat{\bz}=P\bar{\bz}$. Since $P\bb=\vec{\bb}$ and $P$ is orthogonal, $\sum_{i=1}^n \rho \sigma_i l(\bar z_{[i]})+\frac12\|\bar{\bz}-\bb\|^2=\sum_{i=1}^n \rho \sigma_i l(\hat z_i)+\frac12\|\hat{\bz}-\vec{\bb}\|^2$. Thus, $\hat{\bz}$ is the unique optimal solution to \eqref{problem:pav-problem} with $\by=P\bb=\vec{\bb}$, namely, $\hat{\bz}=\bz^*(P\bb)=\bz^*(\vec{\bb})$. Therefore, $\prox_{\rho f}(\bb)=P^\top \bz^*(P\bb)=P^\top \bz^*(\vec{\bb})$.
\end{proof}

\subsection{A Structure-Exploiting PAVA}
The ordered subproblem~\eqref{problem:pav-problem} captures the ranking structure in \(f\), which makes the SRM proximal mapping nonsmooth and nonseparable. To exploit this structure, we develop a PAVA variant tailored to problem~\eqref{problem:pav-problem}. We begin by recalling the classical PAVA~\cite{best2000minimizing}, which solves isotonic optimization problems of the form
\[
\min_{\bz}\quad \sum_{i=1}^n \theta_i(z_i)
\qquad \mathrm{s.t.} \quad
z_1 \le z_2 \le \cdots \le z_n,
\]
where each $\theta_i$ is a univariate convex function.
The algorithm maintains a partition of $\llbracket 1,n\rrbracket$ into consecutive blocks
$[s_1+1:s_2], [s_2+1:s_3], \dots, [s_K+1:s_{K+1}]$, with $s_1=0$ and $s_{K+1}=n$,
where all variables in a block share a common value
\[
v_{[p,q]} \triangleq  \arg\min_{\bar z} \sum_{i=p}^q \theta_i(\bar z).
\]
Two consecutive blocks $[p,q]$ and $[q+1,r]$ are called \emph{in-order} if $v_{[p,q]} \le v_{[q+1,r]}$ and \emph{out-of-order} otherwise.
More generally, a sequence of consecutive blocks is \emph{consecutively out-of-order} if the common values strictly decrease.
PAVA starts by treating each index as a singleton block $[i:i]$.
During iterations, blocks are exclusively merged such that whenever two consecutive blocks are \emph{out-of-order}, they are combined into a single block.
The process repeats until all blocks are \emph{in-order}, at which point the algorithm terminates.

Given the structure of problem \eqref{problem:pav-problem}, Proposition~1 of \cite{xiao2023unified} presents an improved version of the traditional PAVA. Specifically, when encountering consecutively \emph{out-of-order} blocks of length greater than 2, this variant allows for simultaneous merging, rather than performing block-wise merging individually.

Our variant differs from the classical PAVA in two respects: it allows backward pooling after a forward pooling step and delays the evaluation of merged-block minimizers. The local tests that make these two modifications implementable are developed next.

In our setting, $\theta_i(z_i)=\rho\sigma_i\,l(z_i)+\frac12(z_i-y_i)^2$.
For a pooled block \([m,r]\), the associated scalar decision is obtained from
\[
\min_{\bar{z}}\ \sum_{i=m}^{r}\left(\rho\sigma_i\,l(\bar{z})+\frac12(\bar{z}-y_i)^2\right).
\]
By collecting the terms that depend on \(\bar z\), this problem is equivalent to
\begin{equation}
\label{problem:merge-block-problem}
\min_{\bar{z}}\ \theta_{[m,r]}(\bar{z})
\triangleq \rho\,\sigma_{[m,r]}l(\bar{z})+\frac12\bigl(\bar{z}-y_{[m,r]}\bigr)^2,
\end{equation}
where
$
\sigma_{[m,r]}\triangleq \frac{1}{r-m+1}\sum_{i=m}^{r}\sigma_i,
y_{[m,r]}\triangleq \frac{1}{r-m+1}\sum_{i=m}^{r}y_i.
$

The following proposition gives the block-order tests used to decide whether the current block should be pooled further with a neighboring block. For a continuously differentiable loss, these tests use only first-order information.

\begin{proposition}
\label{prop:pava-opt}
Suppose that Assumption~\ref{assp:l-nondecreasing} holds. Consider a block $[m,r]$ in the current partition. Let $[p,m-1]$ and $[r+1,q]$ denote its adjacent predecessor and successor, when present, and suppose that their block values are known. Set
$
v_{[m,r]}\triangleq \arg\min_{\bar z}\theta_{[m,r]}(\bar z),
$
with $\theta_{[m,r]}$ defined in \eqref{problem:merge-block-problem}. For the successor, the \textnormal{\textsc{Forward Condition}} in $[m,r]$ is $v_{[m,r]} > v_{[r+1,q]}$. For the predecessor, the \textnormal{\textsc{Backward Condition}} in $[m,r]$ is $v_{[m,r]} < v_{[p,m-1]}$. The following criteria apply to these forward and backward tests.
\begin{enumerate}
\item For $l(z)=\max\{0,1+z\}$,
    \begin{equation}
    \label{eq:hinge-block-minimizer}
    v_{[m,r]} = \min\!\left\{y_{[m,r]},\ \max\{y_{[m,r]}-\rho\sigma_{[m,r]},-1\}\right\}.
    \end{equation}
    Hence the applicable conditions are checked by direct comparison with the neighboring block values.
\item For a continuously differentiable loss $l$, the \textnormal{\textsc{Forward Condition}} in $[m,r]$ holds if and only if $\theta'_{[m,r]}(v_{[r+1,q]}) < 0$, and the \textnormal{\textsc{Backward Condition}} in $[m,r]$ holds if and only if $\theta'_{[m,r]}(v_{[p,m-1]}) > 0$.
\end{enumerate}
\end{proposition}

\begin{proof}
    Let $w=\rho\sigma_{[m,r]}\ge 0$ and $y=y_{[m,r]}$. For the hinge loss, $v_{[m,r]}$ is the unique minimizer of $\theta_{[m,r]}(z)=w\max\{0,1+z\}+\frac12(z-y)^2$, so the optimality condition is $0\in w\partial\max\{0,1+v_{[m,r]}\}+v_{[m,r]}-y$. If $v_{[m,r]}<-1$, then $v_{[m,r]}=y$, which requires $y\le -1$; if $v_{[m,r]}=-1$, then $y\in[-1,-1+w]$; if $v_{[m,r]}>-1$, then $v_{[m,r]}=y-w$, which requires $y\ge -1+w$. Hence $v_{[m,r]}=\min\{y,\max\{y-w,-1\}\}$, which gives \eqref{eq:hinge-block-minimizer}.
    Substituting this value into the definitions of the \textnormal{\textsc{Forward Condition}} and \textnormal{\textsc{Backward Condition}} gives the stated hinge-loss tests.

    Now suppose that $l$ is continuously differentiable. Then $\theta_{[m,r]}$ is differentiable and strongly convex, with derivative $\theta'_{[m,r]}(z)=\rho\sigma_{[m,r]}l'(z)+z-y_{[m,r]}$. By Theorem~2.13 in~\cite{rockafellar2009variational}, the convexity of $l$ implies that $l'$ is nondecreasing; therefore, $\theta'_{[m,r]}$ is strictly increasing. Since $v_{[m,r]}$ is the unique minimizer of $\theta_{[m,r]}$, we have $\theta'_{[m,r]}(v_{[m,r]})=0$, and thus $\theta'_{[m,r]}(x)<0$ if and only if $x<v_{[m,r]}$, while $\theta'_{[m,r]}(x)>0$ if and only if $x>v_{[m,r]}$. Taking $x=v_{[r+1,q]}$ and $x=v_{[p,m-1]}$ yields the stated forward and backward criteria.
\end{proof}

Proposition~\ref{prop:pava-opt} supplies the local tests used in Algorithm~\ref{alg:pav}. For a newly formed candidate block, the \textnormal{\textsc{Forward Condition}} and \textnormal{\textsc{Backward Condition}} are checked by the criteria in Proposition~\ref{prop:pava-opt}, using either the closed-form hinge test or the derivative test at neighboring block values. Thus, the univariate minimization for the candidate block is solved only after the inner loop terminates.

Algorithm~\ref{alg:pav} summarizes this recursive pooling scheme. In each main-loop iteration, line~5 initializes the current candidate block \([s_k,t-1]\), where \(s_k\) is its first index and \(t\) is the index of the first singleton block to its right. The inner loop in line~6 repeatedly checks the two \emph{out-of-order} conditions from Proposition~\ref{prop:pava-opt}. A forward pooling step in lines~8--9 appends \([t,t]\) to the current candidate block, while a backward pooling step in lines~10--12 merges the candidate block with its predecessor. The flag \(V_{\text{flag}}\), initialized in line~5 and set in line~7, determines whether line~14 computes the merged-block minimizer. If no pooling occurs, line~14 is skipped. Line~15 then moves to the next candidate block. The following proposition formalizes the resulting termination bound, subproblem count, and global optimality guarantee.

\begin{algorithm}[htbp]
\caption{A practical pool adjacent violators algorithm for \eqref{problem:pav-problem}}
\label{alg:pav}
\begin{algorithmic}[1]

\Statex \textbf{Input:}  set $J = \{ [1,1],[2,2],\dots,[n,n]\}$, components $\theta_i(\cdot)$, $i = 1,2,\dots,n$

\FOR{each $[i,i] \in J$}
    \STATE compute the minimizer $v_{[i,i]}$ of $\theta_i(\bar z)$
\ENDFOR

\STATE $k=1$, $j = 1$

\WHILE{$j \leq n$}
    \STATE $t = j + 1$, $V_{\text{flag}} \leftarrow \text{False}$, $s_k = j$
    \WHILE{($t \le n$ and the \textnormal{\textsc{Forward Condition}} in $[s_k,t-1]$ holds) or ($k>1$ and the \textnormal{\textsc{Backward Condition}} in $[s_k,t-1]$ holds)}
        \STATE $V_{\text{flag}} \leftarrow \text{True}$
        \IF{$t \le n$ and the \textnormal{\textsc{Forward Condition}} in $[s_k,t-1]$ holds}
            \STATE $J \leftarrow (J \setminus \{[s_k,t-1],[t,t]\}) \cup \{[s_k,t]\}$, $t \leftarrow t + 1$ \Comment{Forward pooling}
        \ENDIF
        \IF{$k>1$ and the \textnormal{\textsc{Backward Condition}} in $[s_k,t-1]$ holds}
            \STATE $\begin{aligned}[t]
                J &\leftarrow (J \setminus \{[s_{k-1},s_k-1],[s_k,t-1]\})\cup \{[s_{k-1},t-1]\}
            \end{aligned}$
            \STATE $s_k \leftarrow s_{k-1}$, $k \leftarrow k - 1$ \Comment{Backward pooling}
        \ENDIF
    \ENDWHILE

    \IF{$V_{\text{flag}} = \text{True}$}
        \STATE compute the minimizer $v_{[s_k,t-1]}$ of $\sum_{i = s_k}^{t-1} \theta_i(\bar z)$
    \ENDIF

    \STATE $j \leftarrow t,~k \leftarrow k+1$ \Comment{Move to next candidate block}
\ENDWHILE

\FOR{each $[m,r] \in J$}
    \STATE $z_i^* = v_{[m,r]},~\forall i = m, m+1, \dots, r$
\ENDFOR
\STATE \textbf{Output:} optimal solution $\{z_1^*,z_2^*,\dots, z_n^* \}$
\end{algorithmic}
\end{algorithm}

\begin{proposition}
Suppose that Assumption~\ref{assp:l-nondecreasing} holds and that the individual loss $l$ is either the hinge loss or continuously differentiable. Then Algorithm~\ref{alg:pav} performs at most \(2n-1\) iterations of its two while loops in total and solves at most \(2n-1\) univariate optimization subproblems. Furthermore, Algorithm~\ref{alg:pav} is guaranteed to find a global minimizer of problem~\eqref{problem:pav-problem}.
\end{proposition}
\begin{proof}
The main-while loop (lines~4--15) is executed at most $n$ times, since $j$ increases by at least one in each main-loop iteration. In the inner while loop (lines~6--12), each iteration performs at least one merge. Since the algorithm starts from \(n\) singleton blocks and no merge can reduce the number of blocks below one, there are at most \(n-1\) merges, so this inner loop is executed at most \(n-1\) times in total. Therefore, the total number of main-loop and inner-loop iterations is at most \(2n-1\).

Algorithm~\ref{alg:pav} initially solves $n$ singleton subproblems. Afterwards, an additional univariate subproblem is solved only when $V_{\text{flag}}=\text{True}$, which means that the current main-loop iteration contains a merge. The same merge count therefore bounds the number of additional subproblems by \(n-1\). Hence the total number of univariate optimization subproblems solved is at most \(2n-1\).

Finally, by Proposition~\ref{prop:pava-opt}, each forward or backward test is exactly an \emph{out-of-order} test for the current block with its successor or predecessor. Thus, Algorithm~\ref{alg:pav} performs the same block-merging logic as PAVA while delaying the evaluation of the merged-block minimizer until the inner loop terminates. Consequently, the algorithm terminates with an \emph{in-order} partition, and Theorem~2.5 of \cite{best2000minimizing} implies that the resulting vector is a global minimizer of problem~\eqref{problem:pav-problem}.
\end{proof}
\begin{remark}
The preceding proposition gives explicit worst-case bounds of \(2n-1\) for the number of univariate optimization subproblems. In practice, the subproblem count is often smaller. For instance, under the $\nu$-superquantile weights in \eqref{eq:srm_cvar}, if \(n\nu\) is an integer, the number of such subproblems is at most $n+1$. Specifically, after the initial computation of the optimal values for all $n$ singleton blocks, the only possible violation of monotonicity occurs at the boundary between blocks $i$ and $i+1$, where $i$ is the largest index satisfying $\sigma_i=0$ (i.e., $\sigma_i=0$ and $\sigma_{i+1}\neq 0$). Therefore, after a single main-while loop iteration, the block values are guaranteed to be \emph{in-order}. By comparison, in the worst case, Algorithm~2 in \cite{xiao2023unified} may still need to solve up to $2n$ univariate optimization subproblems.
\end{remark}

\section{\texorpdfstring{Generalized Jacobians of the SRM Proximal Mapping}{Generalized Jacobians of the SRM Proximal Mapping}}
\label{sec:hs-jacobian}
Building on the PAVA evaluation of \(\operatorname{prox}_{\rho f}\) in Section~\ref{sec:prox-pava}, this section characterizes generalized derivatives of this mapping needed to construct Newton directions from \eqref{eq:nabla-phik}. Following Han and Sun's construction of computable generalized Jacobians for normal maps over polyhedral sets~\cite{han1997newton}, we refer to the derivative family constructed below as an \textit{HS--Jacobian}. We first recall the required semismoothness notions and then derive \textit{HS--Jacobian} characterizations for the hinge and SC\(^1\) losses.

We begin by recalling the semismoothness and generalized Jacobian notions used below and in the \textnormal{\textsc{Ssn}} analysis of Section~\ref{sec:ssn-linear-algebra} \cite{Qi1993semismooth,gowda2004inverse,li2018efficiently,clarke1990nonsmooth}.

\begin{definition}[G-semismoothness \cite{gowda2004inverse} and Semismoothness \cite{Qi1993semismooth}]
\label{def:semismooth}
Let $F : \Omega \to \mathbb{R}^m$ be a locally Lipschitz continuous mapping on an open set $\Omega \subseteq \mathbb{R}^n$. We say $F$ is \emph{G-semismooth} at $\bx \in \Omega$ with respect to an upper semicontinuous, compact-valued mapping $\mathcal{K} : \Omega \rightrightarrows \mathbb{R}^{m \times n}$ if, as $\bd \to \bs{0}$,
\[
F(\bx + \bd) - F(\bx) - J \bd = o(\|\bd\|), \qquad \forall J \in \mathcal{K}(\bx + \bd).
\]
Moreover, $F$ is said to be \emph{strongly G-semismooth} at $\bx$ with respect to $\mathcal{K}$ if the above remainder is $O(\|\bd\|^2)$.
Furthermore, if $F$ is both (strongly) G-semismooth at $\bx$ with respect to $\mathcal{K}$ and directionally differentiable at $\bx$, then it is called (strongly) \emph{semismooth} at $\bx$ with respect to $\mathcal{K}$.
\end{definition}

We shall also use the standard fact that the proximal mapping of a proper closed convex piecewise linear--quadratic function is piecewise affine~\cite[Proposition~12.30]{rockafellar2009variational}, and therefore strongly semismooth~\cite[Proposition~7.4.7]{facchinei2003finite}.

\begin{definition}[SC$^1$ function, {\cite[Section~7.4.1]{facchinei2003finite}}]\label{def:sc1-fun}
A function \(l:\mathbb{R}\to\mathbb{R}\) is called an SC\(^1\) function if \(l\) is continuously differentiable and its derivative \(l'\) is semismooth with respect to \(\partial l'\). Here, \(\partial l'\) denotes the Clarke generalized Jacobian of \(l'\)~\cite{clarke1990nonsmooth}.
\end{definition}

Many loss functions used in practice are SC\(^1\), e.g., \(C^2\) losses such as the logistic loss \(l(z)=\log(1+e^{z})\), and the smoothed hinge loss introduced below Assumption~\ref{assp:l-nondecreasing}.

\subsection{The Hinge-Loss Case}
We first consider the hinge loss \(l(z)=\max\{0,1+z\}\). Its piecewise linear structure allows a more explicit description of \(\operatorname{prox}_{\rho f}\) and its \textit{HS--Jacobian}~\cite{han1997newton}. For an ordered data vector \(\by\), specializing problem~\eqref{problem:pav-problem} to the hinge loss yields the following formulation
\begin{equation}
\label{problem:hinge-loss-pava}
\min_z\quad \sum_{i=1}^n \left(\rho\sigma_i \max\{1+z_i, 0\} + \frac{1}{2}( z_i-y_i )^2\right) \qquad \mathrm{s.t.}\quad  z_1 \leq z_2 \leq \cdots \leq z_n.
\end{equation}

We partition the ordered data vector \(\by\) relative to the threshold \(-1\) by defining indices \(0\le k_0\le k_1\le n\) such that
\begin{equation}
\label{eq:structure-of-vecb}
y_1 \le \cdots \le y_{k_0} < -1,
\qquad
y_{k_0+1} = \cdots = y_{k_1} = -1,
\qquad
-1 < y_{k_1+1} \le \cdots \le y_n.
\end{equation}
The partition in \eqref{eq:structure-of-vecb} allows some segments to be empty.  We now introduce an auxiliary subproblem and its associated index family $\mathcal{B}^s(\by)$. For \(s\in\llbracket 0,n-1\rrbracket\) and a vector \(\by \in \mathbb{R}^{n-s}\), consider
\begin{equation}
\label{problem:left-part-problem}
\begin{aligned}
\min_{\bar{\bz} \in \mathbb{R}^{n-s}}
\sum_{i=1}^{n-s} \left(\rho\sigma_{s+i} \bar{z}_i +\frac{1}{2}( \bar{z}_i- y_i )^2\right) \qquad
\mathrm{s.t.}\quad
 -1 \leq \bar{z}_1 \leq \bar{z}_2 \leq \cdots \leq \bar{z}_{n-s}.
\end{aligned}
\end{equation}
Because the linear independence constraint qualification (LICQ)~\cite[Definition~12.4]{nocedal2006numerical} trivially holds for \eqref{problem:left-part-problem}, the Lagrange multiplier associated with its optimal solution is unique~\cite[Exercise~12.17]{nocedal2006numerical}.

\begin{definition}
\label{def:b-family-and-solution}
    For \(s\in\llbracket 0,n-1\rrbracket\) and \(\by\in\mathbb{R}^{n-s}\), let \(\bz^*_{(s)}(\by)\in \mathbb{R}^{n-s}\) denote the unique optimal solution to \eqref{problem:left-part-problem}, and let \(\blambda^*_{(s)}(\by)\) denote its unique optimal Lagrange multiplier. Set \((\bz^*_{(s)}(\by))_0\triangleq -1\) and define the active index set
    \[
    \mathcal I^s(\by)\triangleq \{i\in\llbracket 1,n-s\rrbracket\mid (\bz^*_{(s)}(\by))_i=(\bz^*_{(s)}(\by))_{i-1}\}.
    \]
    Define
    \begin{equation}
    \label{eq:b-family}
    \mathcal{B}^s(\by)\triangleq
    \left\{K\subseteq\llbracket 1,n-s\rrbracket\mid
    \operatorname{supp}(\blambda^*_{(s)}(\by))\subseteq K\subseteq \mathcal I^s(\by)
    \right\}.
    \end{equation}
    When \(s=n\), the auxiliary problem is vacuous, and we use the conventions \(\bz^*_{(n)}(\emptyset)=\emptyset\), \(\mathcal I^n(\emptyset)=\emptyset\), and \(\mathcal{B}^n(\emptyset)=\{\emptyset\}\).
\end{definition}

A key observation from the mechanics of Algorithm~\ref{alg:pav} is that the solution to problem~\eqref{problem:hinge-loss-pava} structurally decomposes into distinct segments.
\begin{proposition}
    \label{prop:hinge-pava-solution}
    Let $\bb\in\mathbb{R}^n$, set \(\by=\vec{\bb}\), and let $\bz^*(\by)$ denote the unique solution to problem~\eqref{problem:hinge-loss-pava}. Based on the partition indices $0 \le k_0 \le k_1 \le n$ defined in \eqref{eq:structure-of-vecb}, the first two segments of the solution are
    \[ \bz^*_{\llbracket 1,k_0\rrbracket}(\by) = \by_{\llbracket 1,k_0\rrbracket} < -\mathbf{1}, \qquad \bz^*_{\llbracket k_0 + 1,k_1\rrbracket}(\by) = -\mathbf{1}, \]
    where the equalities and inequalities involving vectors are understood componentwise, and the corresponding segments are empty if $k_0 = 0$ or $k_1 = k_0$, respectively.     
    The remaining tail is represented by the auxiliary problem~\eqref{problem:left-part-problem}: for every $s \in \llbracket k_0,k_1\rrbracket$,
    \[
    \bz^*_{\llbracket s+1,n\rrbracket}(\by) = \bz^*_{(s)}(\by_{\llbracket s+1,n\rrbracket}),
    \]
    where both sides are understood as empty vectors when $s=n$.
\end{proposition}
\begin{proof}
Because the objective in \eqref{problem:hinge-loss-pava} is the sum of a convex hinge-loss term and a strictly convex quadratic term, it is strongly convex and admits a unique minimizer $\bz^*$. By the KKT conditions, there exists a multiplier $\blambda\in\mathbb{R}^{n-1}$ such that
\begin{subequations}
\begin{alignat}{2}
&0 \in \rho\sigma_i\,\partial l(z_i^*) + z_i^* - y_i + \lambda_i-\lambda_{i-1},
&\qquad& i=1,\dots,n, \label{kkt:stationarity} \\
&\lambda_i(z_i^*-z_{i+1}^*) = 0,\quad
\lambda_i\ge 0,\quad z_i^*\le z_{i+1}^*,
&\qquad& i=1,\dots,n-1, \label{kkt:complementary}
\end{alignat}
\end{subequations}
where $l(z)=\max\{1+z,0\}$, $\lambda_0=\lambda_n=0$, and $\partial l(z)=\{0\}$ for $z<-1$, $\partial l(-1)=[0,1]$, and $\partial l(z)=\{1\}$ for $z>-1$.
We first consider the generic case $0<k_0<k_1<n$, in which all three segments are present.

\emph{Prefix segment.} We show that $\bz^*_{\llbracket 1,k_0\rrbracket}=\by_{\llbracket 1,k_0\rrbracket}$. Suppose for contradiction that $z_1^*\ge -1$. Since $k_0>0$, we have $y_1<-1$, meaning $y_1-z_1^*<0$. From \eqref{kkt:stationarity} at $i=1$, we obtain $y_1-z_1^*-\lambda_1 \in \rho\sigma_1\,\partial l(z_1^*)$. Because $\lambda_1\ge0$, the left-hand side is strictly negative, whereas $\rho\sigma_1\,\partial l(z_1^*)\subseteq[0,\rho\sigma_1]$, yielding a contradiction. Hence $z_1^*<-1$, so $\partial l(z_1^*)=\{0\}$, and \eqref{kkt:stationarity} simplifies to $z_1^*=y_1-\lambda_1$.

Suppose further that $\lambda_1>0$. Then by~\eqref{kkt:complementary}, $z_1^*=z_2^* < -1$, yielding $\partial l(z_2^*)=\{0\}$. Substituting into \eqref{kkt:stationarity} for $i=1,2$ gives $\lambda_2 = 2\lambda_1 + (y_2-y_1)\ge 2\lambda_1>0$. Inductively, this implies $\lambda_i>0$ for all $i=1,\dots,n-1$, resulting in $z_1^*=\cdots=z_n^*< -1$. However, \eqref{kkt:stationarity} at $i=n$ then requires $z_n^*=y_n+\lambda_{n-1}>-1$ (since $y_n\ge -1$ and $\lambda_{n-1}>0$), a contradiction. Therefore, $\lambda_1=0$ and $z_1^*=y_1$. Applying this argument successively to $i=2,\dots,k_0$ confirms that $\lambda_i = 0$ and $z_i^* = y_i$ for all $i \le k_0$. Thus, $\bz^*_{\llbracket 1,k_0\rrbracket}=\by_{\llbracket 1,k_0\rrbracket}$.

\emph{Threshold segment.} We show that $z_i^*=-1$ for all $i\in \llbracket k_0+1,k_1\rrbracket$. Since $\lambda_{k_0}=0$, if $z_{k_0+1}^*<-1$, then \eqref{kkt:stationarity} gives $\lambda_{k_0+1}=-(z_{k_0+1}^*+1)>0$, and \eqref{kkt:complementary} forces $z_{k_0+1}^*=z_{k_0+2}^*<-1$. Iterating this argument yields $z_{k_0+1}^*=\cdots=z_n^*<-1$ and $\lambda_i>0$ for all $i\ge k_0+1$, which contradicts \eqref{kkt:stationarity} at $i=n$ because then $z_n^*=y_n+\lambda_{n-1}>-1$. If instead $z_{k_0+1}^*>-1$, then \eqref{kkt:stationarity} gives $0=\rho\sigma_{k_0+1}+z_{k_0+1}^*+1+\lambda_{k_0+1}$, again impossible. Hence $z_{k_0+1}^*=-1$.

Now let $i\in \llbracket k_0+2,k_1\rrbracket$ and assume that $z_{i-1}^*=-1$. Monotonicity rules out $z_i^*<-1$. If $z_i^*>-1$, then $z_{i-1}^*<z_i^*$, so \eqref{kkt:complementary} implies $\lambda_{i-1}=0$, and \eqref{kkt:stationarity} gives $0=\rho\sigma_i+z_i^*+1+\lambda_i$, impossible. Thus $z_i^*=-1$. Induction yields $\bz^*_{\llbracket k_0+1,k_1\rrbracket}=-\mathbf{1}$, and hence $\bz^*_{\llbracket k_0+1,s\rrbracket}=-\mathbf{1}$ for every $s\in \llbracket k_0,k_1\rrbracket$, with the convention that this segment is empty when $s=k_0$.

\emph{Tail representation.} Fix $s\in \llbracket k_0,k_1\rrbracket$. If $s=n$, the assertion is immediate. Assume $s<n$. If $s<k_1$, then the preceding paragraph gives $z_{s+1}^*=-1$; if $s=k_1$, then $z_s^*=-1$ and monotonicity gives $z_{s+1}^*\ge -1$. Hence $\bz^*_{\llbracket s+1,n\rrbracket}\ge -\mathbf{1}$, so it is feasible for problem~\eqref{problem:left-part-problem}. On this auxiliary feasible set, the tail objective in problem~\eqref{problem:hinge-loss-pava} equals the objective of problem~\eqref{problem:left-part-problem} plus the constant $\rho\sum_{i=s+1}^n\sigma_i$.
Since problem~\eqref{problem:left-part-problem} is strictly convex, it has a unique minimizer, say $\bar{\bz}=\bz^*_{(s)}(\by_{\llbracket s+1,n\rrbracket})$. Replacing $\bz^*_{\llbracket s+1,n\rrbracket}$ by $\bar{\bz}$ preserves feasibility, because there is no left boundary if $s=0$, while $z_s^*\le -1\le\bar z_1$ if $s>0$. If $\bz^*_{\llbracket s+1,n\rrbracket}\neq \bar{\bz}$, this replacement strictly lowers the objective in problem~\eqref{problem:hinge-loss-pava}, a contradiction. Therefore, $\bz^*_{\llbracket s+1,n\rrbracket}(\by)=\bz^*_{(s)}(\by_{\llbracket s+1,n\rrbracket})$.

\emph{Degenerate cases.} The remaining degenerate cases follow by minor modifications of the same argument. If $k_0=0<k_1$, the prefix segment is absent, and the proof for the $-1$ segment starts from $\lambda_0=0$. If $k_0=k_1<n$, the middle segment is absent; starting from $\lambda_{k_0}=0$, the same contradiction as above excludes $z_{k_0+1}^*<-1$, and hence $z_i^*\ge -1$ for all $i\ge k_0+1$, so $\bz^*_{\llbracket k_0+1,n\rrbracket}$ coincides with the unique solution of problem~\eqref{problem:left-part-problem} with $s=k_0$. If $k_1=n$, the tail segment is absent when $s=n$, and the decomposition terminates at index $n$; in particular, if $k_0=k_1=n$, then $\by<-\mathbf{1}$, and $\by$ is feasible with zero objective value. Since the objective is nonnegative and strictly convex, $\by$ is the unique optimal solution. Consequently, the stated decomposition remains valid for all $0\le k_0\le k_1\le n$.
\end{proof}

By Proposition~\ref{prop:hinge-pava-solution}, the sorted hinge proximal solution is reduced, after its explicit prefix and threshold segments, to the tail problem~\eqref{problem:left-part-problem}. After completing the square, this tail problem is a Euclidean projection onto a polyhedral set. Hence Lemma~2.1 of~\cite{han1997newton} gives a local affine representation of its solution mapping, which we use to define the hinge-loss \textit{HS--Jacobian} set for \(\prox_{\rho f}\).

Let \(C^m_K\) be the row submatrix indexed by \(K\) of the lower bidiagonal matrix \(C^m\in\R^{m\times m}\) with diagonal \(-1\) and subdiagonal \(1\). For \(s=0,\dots,n-1\) and \(\by\in\R^{n-s}\), using \(\mathcal{B}^s(\by)\) from Definition~\ref{def:b-family-and-solution}, set
\begin{equation}
\label{eq:hinge-hat-s}
\widehat{\mathcal S}^{\,s}(\by)
\triangleq
\left\{
I_{n-s}-(C^{n-s}_{K})^\top
\big(C^{n-s}_{K}(C^{n-s}_{K})^\top\big)^{-1}C^{n-s}_{K}
\ \middle|\ 
K\in\mathcal{B}^s(\by)
\right\}.
\end{equation}
Set \(\widehat{\mathcal S}^{\,n}(\emptyset)\triangleq\{I_0\}\), where \(I_0\) is the \(0\times0\) empty matrix, and use the convention \(\vec{\bx}_{\llbracket n+1,n\rrbracket}=\emptyset\). For \(\bx\in\R^n\), define
\begin{equation}
\label{eq:hinge-jac-s}
\mathcal S^s(\bx)
\triangleq
\widehat{\mathcal S}^{\,s}(\vec{\bx}_{\llbracket s+1,n\rrbracket}),
\qquad s=0,\dots,n.
\end{equation}
Let \(k_0(\bx)\) and \(k_1(\bx)\) be the partition indices obtained by applying \eqref{eq:structure-of-vecb} with \(\by=\vec{\bx}\), and define
\begin{equation}
\label{eq:hinge-jac-no-p}
\mathcal{J}(\bx)
\triangleq
\left\{
\operatorname{BlkD}(I_s,J_s)
\ \middle|\
s\in \llbracket k_0(\bx),k_1(\bx)\rrbracket,\
J_s\in\mathcal S^s(\bx)
\right\}.
\end{equation}
Finally, define the hinge-loss \textit{HS--Jacobian} set for the original variable by
\begin{equation}
\label{eq:hinge-jac}
\mathcal{J}_{HS}(\bx)
\triangleq
\{P^\top JP\mid P\in\mathcal{P}(\bx),\ J\in\mathcal{J}(\bx)\}.
\end{equation}

\begin{theorem}[\textit{HS--Jacobian} for the hinge loss]
\label{thm:hinge-loss-jac}
    For the hinge loss, fix any reference point \(\bb\in\R^n\). There exists a neighborhood \(N(\bb)\) of \(\bb\) such that, for every \(\ba\in N(\bb)\),
    \[
    \mathcal{J}_{HS}(\ba)\subseteq\mathcal{J}_{HS}(\bb),
    \qquad
    \prox_{\rho f}(\ba)
    =
    \prox_{\rho f}(\bb)+H(\ba-\bb),
    \quad \forall H\in\mathcal{J}_{HS}(\ba).
    \]
\end{theorem}
\begin{proof}
Let \(k_0\triangleq k_0(\bb)\) and \(k_1\triangleq k_1(\bb)\).
By Lemma~2.1(ii) of \cite{han1997newton}, for each \(s\in \llbracket k_0,k_1\rrbracket\) with \(s<n\), there exists \(\delta_s>0\) such that any \(\ba^s\in\R^{n-s}\) satisfying \(\|\ba^s-\vec{\bb}_{\llbracket s+1,n\rrbracket}\|_2<\delta_s\) also satisfies \(\mathcal B^s(\ba^s)\subseteq\mathcal B^s(\vec{\bb}_{\llbracket s+1,n\rrbracket})\) and, for all \(\widehat J_s\in\widehat{\mathcal S}^{\,s}(\ba^s)\),
\begin{equation}
\label{eq:hinge-tail-sensitivity}
    \bz^*_{(s)}(\ba^s)
    =
    \bz^*_{(s)}(\vec{\bb}_{\llbracket s+1,n\rrbracket})
    +
    \widehat J_s\bigl(\ba^s-\vec{\bb}_{\llbracket s+1,n\rrbracket}\bigr).
\end{equation}
When \(s=n\), \eqref{eq:hinge-tail-sensitivity} holds trivially by setting \(\delta_n=1\) and \(\widehat{\mathcal S}^{\,n}(\emptyset)=\{I_0\}\).
Let
\[
\begin{aligned}
\delta_{\rm ord}
&\triangleq \min\{\tfrac{\vec b_{i+1}-\vec b_i}{3}\mid i=1,\ldots,n-1,\ \vec b_{i+1}\ne\vec b_i\},\\
\delta_{\rm thr}
&\triangleq \min\{\tfrac{|\vec b_i+1|}{3}\mid i=1,\ldots,n,\ \vec b_i\ne -1\},
\end{aligned}
\]
where an empty minimum is understood as \(+\infty\). Set
\[
    \delta
    \triangleq 
    \min\{
        1,
        \delta_{\rm ord},
        \delta_{\rm thr},
        \min_{s\in\llbracket k_0,k_1\rrbracket}\delta_s
    \}.
\]
Observe that for any \(\ba\in N(\bb)\triangleq\{\bb+\bd\mid\|\bd\|_2<\delta\}\), the choice \(\delta\le\delta_{\rm ord}\) preserves all strict order gaps in \(\vec{\bb}\), and hence \(\mathcal P(\ba)\subseteq\mathcal P(\bb)\). The term \(\delta_{\rm thr}\) prevents all components of \(\vec{\bb}\) outside the \(-1\) segment of \(\vec{\bb}\) from crossing the threshold \(-1\). Hence
\begin{equation}
\label{eq:hinge-index-inclusion}
    s\in \llbracket k_0(\ba),k_1(\ba)\rrbracket
    \quad\Longrightarrow\quad
    s\in \llbracket k_0,k_1\rrbracket.
\end{equation}

Let \(H\in\mathcal J_{HS}(\ba)\) be arbitrary. By the definition of \(\mathcal J_{HS}(\ba)\), choose \(P\in\mathcal P(\ba)\), \(s\in \llbracket k_0(\ba),k_1(\ba)\rrbracket\), and \(J_s\in\mathcal S^s(\ba)\) such that \(H=P^\top\operatorname{BlkD}(I_s,J_s)P\). The inclusion \(\mathcal P(\ba)\subseteq\mathcal P(\bb)\) and implication~\eqref{eq:hinge-index-inclusion} give \(P\in\mathcal P(\bb)\) and \(s\in \llbracket k_0,k_1\rrbracket\). Hence \(P\ba=\vec{\ba}\), \(P\bb=\vec{\bb}\), and
\[
    \|\vec{\ba}_{\llbracket s+1,n\rrbracket}-\vec{\bb}_{\llbracket s+1,n\rrbracket}\|_2
    \le
    \|\vec{\ba}-\vec{\bb}\|_2
    =
    \|P(\ba-\bb)\|_2
    =
    \|\ba-\bb\|_2
    <
    \delta
    \le
    \delta_s.
\]
Since \(J_s\in\mathcal S^s(\ba)=\widehat{\mathcal S}^{\,s}(\vec{\ba}_{\llbracket s+1,n\rrbracket})\), formula~\eqref{eq:hinge-tail-sensitivity}, applied with \(\ba^s=\vec{\ba}_{\llbracket s+1,n\rrbracket}\), gives
\begin{equation}
\label{eq:hinge-tail-sensitivity-applied}
    \bz^*_{(s)}(\vec{\ba}_{\llbracket s+1,n\rrbracket})
    =
    \bz^*_{(s)}(\vec{\bb}_{\llbracket s+1,n\rrbracket})
    +
    J_s\bigl(\vec{\ba}_{\llbracket s+1,n\rrbracket}-\vec{\bb}_{\llbracket s+1,n\rrbracket}\bigr).
\end{equation}
Moreover, \(\mathcal S^s(\ba)\subseteq\mathcal S^s(\bb)\). By Proposition~\ref{prop:hinge-pava-solution}, the same splitting index \(s\) is admissible for both \(\vec{\ba}\) and \(\vec{\bb}\). Combining the prefix identities and tail representations from that proposition with \eqref{eq:hinge-tail-sensitivity-applied} yields
\[
    \bz^*(\vec{\ba})-\bz^*(\vec{\bb})
    =
    \operatorname{BlkD}(I_s,J_s)(\vec{\ba}-\vec{\bb}).
\]
Finally, using the permutation representation of \(\prox_{\rho f}\) established above, we obtain
\[
\begin{aligned}
    \prox_{\rho f}(\ba)-\prox_{\rho f}(\bb)
    =
    P^\top\bigl(\bz^*(\vec{\ba})-\bz^*(\vec{\bb})\bigr) 
    =
    H(\ba-\bb).
\end{aligned}
\]
Since \(\ba\) and \(H\in\mathcal J_{HS}(\ba)\) are arbitrary, the identity in the theorem follows. The facts \(\mathcal P(\ba)\subseteq\mathcal P(\bb)\), \(s\in \llbracket k_0,k_1\rrbracket\), and \(\mathcal S^s(\ba)\subseteq\mathcal S^s(\bb)\) also show, by \eqref{eq:hinge-jac-no-p} and \eqref{eq:hinge-jac}, that \(H\in\mathcal J_{HS}(\bb)\). Hence \(\mathcal J_{HS}(\ba)\subseteq\mathcal J_{HS}(\bb)\).
\end{proof}

\subsection{\texorpdfstring{The SC$^1$-Loss Case}{The SC1-Loss Case}}
We next assume that the individual loss is SC$^1$ and derive an \textit{HS--Jacobian} for $\prox_{\rho f}$. Throughout this subsection, we use \textit{HS--Jacobian} in an extended sense: its projector component follows the Han--Sun active-set construction, while its diagonal component is induced by generalized derivatives of \(l'\). We recast problem~\eqref{problem:pav-problem} as
\begin{equation}
\label{problem:smooth-pav-problem}
\begin{aligned}
\min_\bz  ~\sum_{i=1}^n \left(\rho\sigma_i l(z_i)+\frac{1}{2}( z_i-y_i )^2\right) \qquad
\mathrm{s.t.} \quad  B\bz \leq 0,
\end{aligned}
\end{equation}
where \(B\in\R^{(n-1)\times n}\) is the first-order difference matrix whose
\(i\)-th row is \(\bs{e}_i^\top-\bs{e}_{i+1}^\top\). For any
\(K\subseteq\llbracket 1,n-1\rrbracket\), let \(B_K\) denote the row submatrix of \(B\)
indexed by \(K\).
By the same LICQ argument, the Lagrange multiplier associated with the optimal solution of \eqref{problem:smooth-pav-problem} is unique.
For $\by=\vec{\bb}$, the Lagrangian of \eqref{problem:smooth-pav-problem} is
\[
L(\bz;\blambda) =  \rho \boldsymbol{\sigma}^\top \ell(\bz) + \frac{1}{2}\Vert \bz-\vec{\bb} \Vert^2 + \blambda^\top B \bz,
\]
yielding the KKT system
\begin{equation}
\label{eq:kkt-condition}
\left\{ \begin{aligned}
    &\nabla_{\bz} L(\bz;\blambda) =\rho\boldsymbol{\sigma} \circ \nabla \ell(\bz) +  \bz-\vec{\bb} + B^\top\blambda=0, \\
    &\blambda \circ (B\bz) = \boldsymbol{0}, \quad \blambda \geq 0, \quad B\bz \leq 0,
\end{aligned}\right.
\end{equation}
where $\circ$ denotes the Hadamard product.

Analogous to Definition~\ref{def:b-family-and-solution}, we characterize the index family for problem~\eqref{problem:smooth-pav-problem} as follows.

\begin{definition}
\label{def:sc1-b}
    For any \(\by \in \mathbb{R}^n\), let \(\bz^*(\by)\) and \(\blambda^*(\by)\) denote, respectively, the unique optimal solution and the unique optimal Lagrange multiplier of problem~\eqref{problem:smooth-pav-problem}. Define the active index set and the associated index family by
    \[
    \begin{aligned}
    I(\by) &\triangleq  \left\{i \in \llbracket 1,n-1\rrbracket \mid z_i^*(\by) = z_{i+1}^*(\by)\right\},\\
    \mathcal{B}(\by)&\triangleq
    \left\{K\subseteq\llbracket 1,n-1\rrbracket\mid
    \operatorname{supp}(\blambda^*(\by))\subseteq K\subseteq I(\by)
    \right\}.
    \end{aligned}
    \]
\end{definition}

Using \(\mathcal{B}(\by)\) from Definition~\ref{def:sc1-b}, define the projection-matrix family
\begin{equation}
\label{eq:mathcal-q}
    \mathcal{Q}(\by)\triangleq
    \{ I_n - B_K^\top (B_K B_K^\top)^{-1}B_K \mid K\in \mathcal{B}(\by)\},
\end{equation}
and the diagonal derivative family
\begin{equation}
\label{eq:mathcal-a}
    \mathcal{A}(\by)\triangleq
    \{\operatorname{diag}(\rho\sigma_1h_1,\dots,\rho\sigma_nh_n)\mid h_i \in \partial l'(z_i^*(\by)),~ i\in\llbracket 1,n\rrbracket\}.
\end{equation}
We note that $B_K$ has full row rank, and every \(Q\in\mathcal Q(\by)\) is an orthogonal projector. Moreover, the convexity of $l$ implies that \(l'\) is nondecreasing. Since \(l'\) is semismooth, it is locally Lipschitz, and \cite[Proposition~2.3(a)]{jiang1995local} implies \(h_i\ge0\) for \(h_i\in\partial l'(z_i^*(\by))\). Consequently, for every \(P\in\mathcal P(\bx)\), each \(A\in\mathcal A(P\bx)\) is positive semidefinite.
Finally, define the SC$^1$-loss \textit{HS--Jacobian} set by
\begin{equation}
\label{eq:sc1-jac}
    \mathcal{J}_{HS}(\bx) \triangleq
    \{P^\top (I_n+QA)^{-1} Q P \mid  P \in \mathcal{P}(\bx),~Q \in \mathcal{Q}(P\bx),~A \in \mathcal{A}(P\bx)\}.
\end{equation}
We first verify that $I_n+QA$ in the above definition is indeed nonsingular. Note that any $A\in {\cal A}(P\bx)$ is positive semidefinite since \(\rho\sigma_i\ge0\) and \(h_i\ge0\). Now, suppose $(I_n+QA)\bx=0$. Then \(\bx=-QA\bx=Q(-A\bx)\), and hence \(Q\bx=Q^2(-A\bx)=Q(-A\bx)=\bx\). Consequently, $\|\bx\|^2 = -\bx^\top QA\bx = -\bx^\top A\bx \le 0$, which forces $\bx=0$. 
For every admissible choice of \(P,Q,A\) in \eqref{eq:sc1-jac}, each element of $ \mathcal{J}_{HS}(\bx)$ is symmetric positive semidefinite since
\begin{equation}
\label{eq:sc1-jac-psd}
P^\top (I_n+QA)^{-1}Q P
= P^\top Q (I_n+Q AQ)^{-1}Q P \succeq0.
\end{equation}
The positive semidefiniteness holds since  \(Q\) is an orthogonal projector and \(I_n+QAQ\succ0\).
To justify the equality, fix any \(\bu\) and set \(\bv=(I_n+QA)^{-1}Q\bu\). Then \(\bv=Q(\bu-A\bv)\in\operatorname{Range}(Q)\), so \(Q\bv=\bv\). Hence \((I_n+QAQ)\bv=Q\bu\), which gives \((I_n+QA)^{-1}Q=(I_n+QAQ)^{-1}Q\). Since \(Q=Q^\top=Q^2\) and \(Q\) commutes with \(I_n+QAQ\), we further have \((I_n+QAQ)^{-1}Q=Q (I_n+QAQ)^{-1}Q\).
Moreover,
\begin{equation}
\label{eq:sc1-jac-norm}
\|(I_n+QA)^{-1}Q\|
=\|(I_n+QAQ)^{-1}Q\|\le 1.
\end{equation}

The next result extends the resulting \textit{HS--Jacobian} characterization for SC$^1$ losses. Its proof combines the Han--Sun active-set construction with the semismooth expansion of \(l'\).

\begin{theorem}[\textit{HS--Jacobian} for the SC$^1$ loss]
\label{thm:sc1-jacobian}
Fix any reference point $\bb\in\mathbb R^n$. 
    Suppose Assumption~\ref{assp:l-nondecreasing} holds and the individual loss is an SC$^1$ function. Then the following statements hold.
    \begin{enumerate}
        \item There exists a neighborhood $N(\bb)$ of $\bb$ such that for any $\ba \in N(\bb)$ and any $P\in \mathcal{P}(\ba)$ defined in \eqref{eq:perm-mat}, we have $\mathcal{B}(P \ba)\subseteq \mathcal{B}(P \bb)$ and $\mathcal{Q}(P \ba)\subseteq \mathcal{Q}(P \bb)$.
        \item
        The set-valued mapping $\mathcal{J}_{HS}$ defined in \eqref{eq:sc1-jac} is upper semicontinuous at $\bb$, and for any $\ba$ sufficiently close to $\bb$,
        \[\prox_{\rho f}(\ba) = \prox_{\rho f}(\bb) + H(\ba-\bb)+o(\|\ba-\bb\|), \quad \forall H \in \mathcal{J}_{HS}(\ba).\]
        Furthermore, if $l'$ is strongly semismooth with respect to $\partial l'$, the expansion tightens to $O(\|\ba-\bb\|^2)$.
    \end{enumerate}
\end{theorem}

\begin{proof}
        
       (1) There exists a neighborhood \(N_0(\bb)\) such that
        \(\mathcal P(\ba)\subseteq\mathcal P(\bb)\) for all
        \(\ba\in N_0(\bb)\). Hence \(P\ba=\vec{\ba}\) and \(P\bb=\vec{\bb}\)
        whenever \(P\in\mathcal P(\ba)\). Fix \(P\in\mathcal P(\bb)\) and set
        \(\bar{\by}=P\bb\). For problem~\eqref{problem:smooth-pav-problem},
        the solution map \(\bz^*(\cdot)\) is the proximal mapping of
        \(\rho\boldsymbol{\sigma}^{\top}\ell+\delta_{\{B\bz\le0\}}\), and is
        therefore continuous. For any active set \(K\), the gradients of the
        active constraints are the rows of \(B_K\). Since every such row
        submatrix of the first-order difference matrix has full row rank, LICQ
        holds. Hence the associated multiplier is unique.
        Moreover, \(\blambda^*(\cdot)\) is continuous at \(\bar{\by}\). Indeed,
        along any sequence \(\by^j\to\bar{\by}\), the finiteness of active sets
        implies that each subsequence admits a further subsequence with a fixed
        active set \(K\). On this further subsequence, the stationarity equation,
        the full row rank of \(B_K\), and
        \(\bz^*(\by^j)\to\bz^*(\bar{\by})\) yield bounded multipliers. Every
        multiplier limit satisfies the KKT system at \(\bar{\by}\) and hence, by
        multiplier uniqueness, equals \(\blambda^*(\bar{\by})\), which proves the
        claimed continuity.

        By continuity of \(\bz^*(\cdot)\) and \(\blambda^*(\cdot)\), and by
        finiteness of \(\mathcal P(\bb)\), we can shrink the neighborhood
        uniformly so that, for all \(\ba\in N(\bb)\) and
        \(P\in\mathcal P(\ba)\),
        \[
        \operatorname{supp}(\blambda^*(P\bb))
        \subseteq
        \operatorname{supp}(\blambda^*(P\ba)),
        \qquad
        I(P\ba)\subseteq I(P\bb).
        \]
        Thus, for any \(K\in\mathcal B(P\ba)\),
        \[
        \operatorname{supp}(\blambda^*(P\bb))
        \subseteq
        \operatorname{supp}(\blambda^*(P\ba))
        \subseteq K\subseteq I(P\ba)\subseteq I(P\bb),
        \]
        which gives \(K\in\mathcal B(P\bb)\). This proves
        \(\mathcal B(P\ba)\subseteq\mathcal B(P\bb)\), and
        \(\mathcal Q(P\ba)\subseteq\mathcal Q(P\bb)\) follows from
        \eqref{eq:mathcal-q}.

        (2) Fix any $\ba \in N(\bb)$ and any $P \in \mathcal{P}(\ba)$. Then $P\ba=\vec{\ba}$ and $P\bb=\vec{\bb}$. Let $Q \in \mathcal{Q}(\vec{\ba})$ and $A \in \mathcal{A}(\vec{\ba})$ be arbitrary. By part (1), there exists an index set $K \in \mathcal{B}(\vec{\ba}) \subseteq \mathcal{B}(\vec{\bb})$ such that $Q = I_n - B_K^\top (B_K B_K^\top)^{-1}B_K$. For each $\by \in \{\vec{\ba},\vec{\bb}\}$, since $K \in \mathcal{B}(\by)$, we have \(\operatorname{supp}(\blambda^*(\by)) \subseteq K \subseteq I(\by)\). Let $\blambda_K^*(\by)$ denote the subvector indexed by $K$. Using the KKT conditions \eqref{eq:kkt-condition} and the active relation $B_K \bz^*(\by)=\bs{0}$, we obtain $\bz^*(\by)=\by-\rho(\boldsymbol{\sigma}\circ\nabla\ell(\bz^*(\by)))-B_K^\top\blambda_K^*(\by)$ and $B_K\by=\rho B_K(\boldsymbol{\sigma}\circ\nabla\ell(\bz^*(\by)))+B_KB_K^\top\blambda_K^*(\by)$. Eliminating $\blambda_K^*(\by)$ yields the projection formula
        \begin{equation}
        \label{eq:sc1-q-representation}
        \bz^*(\by)=Q\by-\rho Q(\boldsymbol{\sigma}\circ\nabla\ell(\bz^*(\by))), \qquad \by \in \{\vec{\ba},\vec{\bb}\}.
        \end{equation}
        Subtracting the two identities in \eqref{eq:sc1-q-representation} and applying Definition~\ref{def:semismooth} componentwise to the SC$^1$ loss in Definition~\ref{def:sc1-fun}, we deduce that
        \begin{equation}
        \label{eq:sc1-o-equation}
        \bz^*(\vec{\ba}) - \bz^*(\vec{\bb})= Q(\vec{\ba}-\vec{\bb}) - Q \left(A (\bz^*(\vec{\ba}) - \bz^*(\vec{\bb})) +  o(\|\bz^*(\vec{\ba})-\bz^*(\vec{\bb})\|)\right).
        \end{equation}
Rearranging \eqref{eq:sc1-o-equation}, and using \eqref{eq:sc1-jac-norm} and the nonexpansiveness of \(\bz^*(\cdot)\), yields
        \[
        \bz^*(\vec{\ba})-\bz^*(\vec{\bb}) = (I_n+QA)^{-1}Q(\vec{\ba}-\vec{\bb}) + o(\|\vec{\ba}-\vec{\bb}\|).\] 
        Since $\vec{\ba}-\vec{\bb}=P(\ba-\bb)$ and Lemma~\ref{lem:prox-permutation} gives $\prox_{\rho f}(\ba)=P^\top\bz^*(\vec{\ba})$ and $\prox_{\rho f}(\bb)=P^\top\bz^*(\vec{\bb})$, it follows that 
             \[
        \prox_{\rho f}(\ba)-\prox_{\rho f}(\bb)=P^\top(I_n+QA)^{-1}QP(\ba-\bb)+o(\|\ba-\bb\|).
    \] 
    Because $P \in \mathcal{P}(\ba)$, $Q \in \mathcal{Q}(P\ba)$, and $A \in \mathcal{A}(P\ba)$ were arbitrary, the above relation proves that $\prox_{\rho f}(\ba)=\prox_{\rho f}(\bb)+H(\ba-\bb)+o(\|\ba-\bb\|)$ for all $H \in \mathcal{J}_{HS}(\ba)$.

We next verify the upper semicontinuity of $\mathcal{J}_{HS}$ at $\bb$. It suffices to show that the outer limit of $\mathcal{J}_{HS}$ at $\bb$ is contained in $\mathcal{J}_{HS}(\bb)$. Let $\ba^k\to\bb$, $H^k\in\mathcal{J}_{HS}(\ba^k)$, and $H^k\to H$. For all sufficiently large $k$, the first assertion gives $\mathcal{P}(\ba^k)\subseteq\mathcal{P}(\bb)$ and $\mathcal{Q}(P^k\ba^k)\subseteq\mathcal{Q}(P^k\bb)$, where
\[
H^k=(P^k)^\top(I_n+Q^kA^k)^{-1}Q^kP^k,\quad
P^k\in\mathcal{P}(\ba^k),\quad
Q^k\in\mathcal{Q}(P^k\ba^k),\quad
A^k\in\mathcal{A}(P^k\ba^k).
\]
Since the admissible permutation matrices and projection matrices form finite sets, we may pass to a subsequence along which $P^k=P$ and $Q^k=Q$. Then $P\in\mathcal{P}(\bb)$ and $Q\in\mathcal{Q}(P\bb)$. The mapping $\by\mapsto\bz^*(\by)$ is nonexpansive, because it is the proximal mapping of a proper closed convex function. Hence the upper semicontinuity and compactness of $\partial l'$ imply that $\by\rightrightarrows\mathcal{A}(\by)$ is compact-valued and upper semicontinuous. Passing to a further subsequence if necessary, we have $A^k\to A$ for some $A\in\mathcal{A}(P\bb)$. Since $I_n+QA$ is nonsingular and the map $A\mapsto P^\top(I_n+QA)^{-1}QP$ is continuous, we obtain
\[
H=P^\top(I_n+QA)^{-1}QP\in\mathcal{J}_{HS}(\bb).
\]
This proves the upper semicontinuity of $\mathcal{J}_{HS}$ at $\bb$.

        Finally, if \(l'\) is strongly semismooth with respect to \(\partial l'\), then the componentwise remainder in \eqref{eq:sc1-o-equation} improves from \(o(\|\bz^*(\vec{\ba})-\bz^*(\vec{\bb})\|)\) to \(O(\|\bz^*(\vec{\ba})-\bz^*(\vec{\bb})\|^2)\). Applying \eqref{eq:sc1-jac-norm} and the nonexpansiveness of \(\bz^*(\cdot)\) gives the strengthened \(O(\|\ba-\bb\|^2)\) expansion.
\end{proof}

\section{\texorpdfstring{Semismooth Newton Method for the \textnormal{\textsc{ripALM}} Subproblem}{Semismooth Newton Method for the ripALM Subproblem}}
\label{sec:ssn-linear-algebra}

We now complete the inner solver for the \textnormal{\textsc{ripALM}} \(\bu\)-subproblem~\eqref{problem:inexact-u-subproblem}. The first subsection uses \eqref{eq:nabla-phik} and the results of Sections~\ref{sec:prox-pava}--\ref{sec:hs-jacobian} to obtain a generalized Jacobian of \(\nabla\varphi_k\), state the \textnormal{\textsc{Ssn}} iteration, and record its convergence guarantee. The second subsection derives the linear algebra used to compute the Newton directions efficiently.

\subsection{Generalized Jacobian and Newton Iteration}

We first record the generalized Jacobian for the gradient mapping induced by the proximal results above. The representation \eqref{eq:nabla-phik} decomposes \(\nabla\varphi_k\) into affine terms, \(\prox_{\rho_k f}\), and \(\prox_{\rho_k g}\). The \textit{HS--Jacobian} characterizations in Section~\ref{sec:hs-jacobian} give the required first-order expansion for \(\prox_{\rho_k f}\). Moreover, since \(g=\lambda\|\bw\|_1\), the mapping \(\operatorname{prox}_{\rho_k g}\) is piecewise affine and hence strongly semismooth. A chain-rule argument then yields the following semismoothness result for \(\nabla\varphi_k\). For notational simplicity, throughout this section we fix an outer iteration \(k\) and write
\((\bw,\bz,\rho,\beta)=(\bw_k,\bz_k,\rho_k,\beta_k)\),
while keeping the previous dual iterate as \(\bu_k\).

\begin{theorem}
\label{thm:ssn-all}
    Suppose Assumption~\ref{assp:l-nondecreasing} holds. Fix \(\bw,\bz,\bu_k\), \(\rho>0\), and \(\beta>0\). Define the set-valued mapping
    \begin{equation}
    \label{eq:nabla-phi-jac}
    \mathcal{K}_{\nabla \varphi}(\bu;\bw,\bz)
    \triangleq  \left\{ \rho \bigl(D W D^\top + V \bigr) + \frac{\beta}{\rho} I_n
    \;\middle|\;
    \begin{aligned}
    W &\in \partial \prox_{\rho g}(\bw - \rho D^\top \bu),\\
    V &\in \mathcal{J}_{HS}(\rho \bu + \bz)
    \end{aligned}
    \right\},
    \end{equation}
    where $\mathcal{J}_{HS}$ is given by \eqref{eq:hinge-jac} for the hinge loss and by \eqref{eq:sc1-jac} for general SC$^1$ losses. Then the following statements hold.
    \begin{enumerate}
        \item If the individual loss is an SC$^1$ function, then $\nabla \varphi$ defined in~\eqref{eq:nabla-phik} is G-semismooth at $\bu$ with respect to $\mathcal{K}_{\nabla \varphi}$.
        \item If the individual loss is the hinge loss, the smoothed hinge loss, or a $C^2$ function with a locally Lipschitz continuous Hessian, then $\nabla \varphi$ is strongly semismooth at $\bu$ with respect to $\mathcal{K}_{\nabla \varphi}$.
    \end{enumerate}
\end{theorem}

\begin{proof}
If the individual loss is an SC$^1$ function, then Theorem~\ref{thm:sc1-jacobian} shows that $\prox_{\rho f}$ is G-semismooth at $\rho\bu+\bz$ with respect to $\mathcal{J}_{HS}$. Moreover, $\prox_{\rho g}$ is piecewise affine and hence strongly semismooth, in particular G-semismooth, with respect to $\partial \prox_{\rho g}(\cdot)$. Since both $\mathcal{J}_{HS}$ and $\partial \prox_{\rho g}(\cdot)$ are upper semicontinuous with compact images, the same holds for $\mathcal{K}_{\nabla \varphi}$ in \eqref{eq:nabla-phi-jac}. The representation 
\[\nabla \varphi(\bu)=\prox_{\rho f}(\rho\bu+\bz)-D\prox_{\rho g}(\bw-\rho D^\top\bu)+\tfrac{\beta}{\rho}(\bu-\bu_k)-\bc\] then yields, by a standard chain rule for G-semismooth mappings analogous to \cite[Theorem~7.5.17]{facchinei2003finite}, that $\nabla \varphi$ is G-semismooth at $\bu$ with respect to $\mathcal{K}_{\nabla \varphi}$.

For the stronger claim, we first verify the \textit{HS--Jacobian} \(O(\|\ba-\bb\|^2)\) expansion of \(\prox_{\rho f}\) with respect to the same \(\mathcal{J}_{HS}\) entering \(\mathcal{K}_{\nabla\varphi}\). This expansion follows from Theorem~\ref{thm:hinge-loss-jac} for the hinge loss and from the strengthened part of Theorem~\ref{thm:sc1-jacobian} for the smoothed hinge and \(C^2\) cases, since \(l'\) is respectively piecewise affine or strongly semismooth.

It remains to verify directional differentiability in the cases where the statement claims strong semismoothness. For the hinge and smoothed hinge losses, \(f\) is piecewise linear--quadratic by the finite order-cone and scalar-breakpoint partitions. Hence Proposition~12.30 of~\cite{rockafellar2009variational} implies that \(\prox_{\rho f}\) is piecewise affine and thus directionally differentiable. For \(C^2\) losses with locally Lipschitz continuous Hessian, directional differentiability of the ordered proximal solution mapping follows from~\cite[Section~7.3 and Proposition~7.1]{bonnans1998sensitivity}. Combining these directional differentiability facts with \textit{HS--Jacobian} \(O(\|\ba-\bb\|^2)\) expansion of \(\prox_{\rho f}\) with respect to the same \(\mathcal{J}_{HS}\)
shows that \(\prox_{\rho f}\) is strongly semismooth with respect to \(\mathcal{J}_{HS}\) for the hinge loss, the smoothed hinge loss, and the stated \(C^2\) losses. The same affine chain rule, together with the piecewise affine \(\prox_{\rho g}\) and the affine term \(\frac{\beta}{\rho}(\bu-\bu_k)\), then gives the claimed strong semismoothness of \(\nabla\varphi\) with respect to \(\mathcal{K}_{\nabla\varphi}\).
\end{proof}

Before stating the algorithm, we record the descent observation used by the Armijo line search.
For every selected \(U\in\mathcal K_{\nabla\varphi}(\bu;\bw,\bz)\), we have
\[
U=\rho(DWD^\top+V)+\frac{\beta}{\rho}I_n\succeq \frac{\beta}{\rho}I_n,
\]
since \(W\succeq0\) for the soft-thresholding mapping \(\prox_{\rho g}\) and \(V\succeq0\) for the selected \textit{HS--Jacobian} element, which follows from the projection representation in Theorem~\ref{thm:hinge-loss-jac} for the hinge loss and from \eqref{eq:sc1-jac-psd} for the SC\(^1\) loss.
Consequently, the exact Newton direction is a descent direction whenever \(\nabla\varphi(\bu)\ne0\).
The residual parameter \(c_g\) in \eqref{eq:ssn-subproblem} is chosen sufficiently small so that the computed inexact Newton direction remains a descent direction; this direction is then globalized by the Armijo line search.
Combining this observation with the generalized Jacobian \eqref{eq:nabla-phi-jac}, we outline the \textnormal{\textsc{Ssn}} method in Algorithm~\ref{alg:ssn-dual}.

\begin{algorithm}[t]
\caption{A semismooth Newton method for the $\bu$-subproblem \eqref{problem:inexact-u-subproblem}}
\label{alg:ssn-dual}
\begin{algorithmic}[1]
\Statex \textbf{Input:} $\bu^0=\bu_k, \bw=\bw_k, \bz=\bz_k, \rho > 0,\theta\in(0,1), \alpha \in (0,1], c_g \in (0,1)$ sufficiently small, $c_l \in (0, 1/2)$
\STATE Set $i=0$
\WHILE{$\bu^{i}$ does not satisfy the inexactness criterion in Algorithm~\ref{alg:rALM}}
    \STATE Select $U^i \in \mathcal{K}_{\nabla \varphi}(\bu^i;\bw,\bz)$ from \eqref{eq:nabla-phi-jac}, and find $\bv^i$ satisfying
    \vspace{-0.7\baselineskip}
    \begin{equation}
    \label{eq:ssn-subproblem}
       \|U^i \bv^i + \nabla \varphi (\bu^i)\| \leq \min(c_{g}, \|\nabla \varphi (\bu^i)\|^{1+\alpha})
    \end{equation}
    \STATE Set $\eta^i = \theta^{m^i}$, where $m^i$ is the smallest nonnegative integer $m$ satisfying
    $\varphi (\bu^i + \theta^m \bv^i) \leq \varphi (\bu^i) + c_l \theta^m \langle \nabla \varphi (\bu^i), \bv^i \rangle$
    \STATE Set $\bu^{i+1} = \bu^i + \eta^i \bv^i$
    \STATE $i = i + 1$
\ENDWHILE
\end{algorithmic}
\end{algorithm}

The following convergence guarantee follows from standard line-search semismooth Newton theory~\cite[Proposition~3.3 and Theorem~3.4]{zhao2010newton}, \cite[Theorem~3]{li2018efficiently}. The global convergence part uses the smooth strong convexity of \(\varphi\) and the Armijo line search, while the local rate follows from the G-semismoothness in Theorem~\ref{thm:ssn-all} and the uniform nonsingularity of the Newton matrices via the Newton-map convergence criterion~\cite[Theorem~4]{klatte2018approximations}.

\begin{theorem}
Suppose that Assumption~\ref{assp:l-nondecreasing} holds, that the individual loss is either the hinge loss or an SC\(^1\) function, and that the parameter \(c_g\) in Algorithm~\ref{alg:ssn-dual} is chosen sufficiently small. Let \(\overline{\bu}_{k+1}\) be the unique minimizer of \eqref{problem:inexact-u-subproblem}. Then the sequence \(\{\bu^i\}\) generated by Algorithm~\ref{alg:ssn-dual} globally converges to \(\overline{\bu}_{k+1}\).

Moreover, the Armijo line search eventually accepts the unit step, and the convergence is locally superlinear. In the strongly semismooth cases of Theorem~\ref{thm:ssn-all}, this rate improves to
\(
\|\bu^{i+1}-\overline{\bu}_{k+1}\|
=O\bigl(\|\bu^i-\overline{\bu}_{k+1}\|^{1+\alpha}\bigr).
\)
\end{theorem}

\subsection{Efficient Implementation of the Newton System}

To avoid the prohibitive cost of directly solving the dense $n \times n$ Newton system \eqref{eq:ssn-subproblem}, we exploit the sparse structure of $\mathcal{K}_{\nabla \varphi}$ and the active sets. Iteration indices are omitted for brevity.

Recall from \eqref{eq:nabla-phi-jac} that any generalized Jacobian $U \in \mathcal{K}_{\nabla \varphi}(\bu)$ selected in Algorithm~\ref{alg:ssn-dual} takes the form $U = \rho(DWD^\top + V) + \frac{\beta}{\rho} I_n$. By Theorems~\ref{thm:hinge-loss-jac} and~\ref{thm:sc1-jacobian}, the matrix $V \in \mathcal{J}_{HS}(\rho \bu + \bz)$ can be explicitly written as $V = P^\top J P$ for some permutation matrix $P \in \mathcal{P}(\rho \bu + \bz)$. For the hinge loss, \(J\in\mathcal J(P(\rho\bu+\bz))\) is selected from \eqref{eq:hinge-jac-no-p}. For the SC$^1$ loss, \(J=(I_n+QA)^{-1} Q\).

Letting $\br = -\nabla \varphi(\bu)$ and substituting the decomposition of $U$ into the Newton system \eqref{eq:ssn-subproblem}, we premultiply the equation by $P$ to obtain the permuted system for $\bar{\bv} = P \bv$ as follows.
\begin{equation}
\label{eq:sort-ssn-eq}
    \left(\rho (PD)W(PD)^\top + \rho J + \frac{\beta}{\rho} I_n \right) \bar{\bv} = P \br.
\end{equation}
The original Newton direction is subsequently recovered via $\bv = P^\top \bar{\bv}$.

Following \cite{wu2023convex}, we efficiently compute $(PD)W(PD)^\top$ by selecting an admissible diagonal matrix $W \in \partial \prox_{\rho g}(\bw-\rho D^\top \bu)$. Specifically, we set $W_{jj} = 1$ if $j \in \mathcal{C}$ and $W_{jj} = 0$ otherwise, where $\mathcal{C} \triangleq  \{i \in \llbracket 1,d\rrbracket \mid |(\bw-\rho D^\top\bu)_i| > \rho\lambda\}$ is the active set with cardinality $q = |\mathcal{C}|$. Letting $\bar{D} \in \mathbb{R}^{n\times q}$ be the submatrix of $PD$ consisting of the columns indexed by $\mathcal{C}$, we achieve the exact factorization $(PD)W(PD)^\top = \bar{D}\bar{D}^\top$. This algebraic reduction decreases the formation cost from $O(n^2 d)$ to $O(n^2 q)$, which is highly advantageous when $q \ll d$. When the system is solved by CG, the product with \(\bar D\bar D^\top\) is computed as \(\bar D(\bar D^\top\bv)\) in \(O(nq)\) operations without forming the \(n\times n\) matrix.

By defining $\bar{J} = \rho J + \beta \rho^{-1} I_n$, the linear system \eqref{eq:sort-ssn-eq} condenses to
\[
(\rho \bar{D} \bar{D}^\top + \bar{J})\bar{\bv} = P\br.
\]
Depending on the problem scale and the active set size $q$, we choose the linear solver as follows.
\begin{enumerate}
\item[(a)] For very large-scale instances, we solve the Newton system by the conjugate gradient method. The matrix \(J\) is positive semidefinite in both cases. For the hinge loss, this follows from the projection blocks in \eqref{eq:hinge-hat-s}--\eqref{eq:hinge-jac-no-p}. For the SC\(^1\) loss, it follows from \eqref{eq:sc1-jac-psd} and the orthogonality of the permutation matrix \(P\). Hence the coefficient matrix in \eqref{eq:sort-ssn-eq} is symmetric positive definite.
\item [(b)] Otherwise, when $q \geq \vartheta n$ for a given threshold $\vartheta \in (0,1]$, we explicitly form the coefficient matrix and employ the Cholesky factorization.
\item [(c)] Otherwise, if $q < \vartheta n$, as typically occurs once the sparse solution is identified, the Sherman-Morrison-Woodbury (SMW) formula gives
    \begin{equation}
        \label{eq:smw-formula-v}
        \bar{\bv} = \bar{J}^{-1}P\br - \bar{J}^{-1} \bar{D} \left(\rho^{-1}I_{q} + \bar{D}^\top \bar{J}^{-1} \bar{D}\right)^{-1} \bar{D}^\top \bar{J}^{-1}P\br.
    \end{equation}
\end{enumerate}
The key operation in the SMW formula is applying $\bar{J}^{-1}$. The remaining operations are products with \(\bar D\) and the solution of the small \(q\times q\) system \(\rho^{-1}I_{q}+\bar D^\top\bar J^{-1}\bar D\). Although the underlying Jacobian structures differ significantly between the hinge and SC$^1$ losses, we show below that $\bar{J}^{-1}$ admits an $O(n)$ application in both cases.

\subsubsection{Hinge Loss}

For the hinge loss, let $\bx=\rho\bu+\bz$ and $k_0=k_0(\bx)$. We evaluate a generalized Jacobian $J \in \mathcal{J}(\bx)$ in \eqref{eq:hinge-jac-no-p} with \(s=k_0\). Here we take the admissible choice \(K=\mathcal I^{k_0}\!\left(\vec{\bx}_{\llbracket k_0+1,n\rrbracket}\right)\), namely the full active set, which belongs to \(\mathcal B^{k_0}\!\left(\vec{\bx}_{\llbracket k_0+1,n\rrbracket}\right)\). This yields the block-diagonal matrix
\[
\bar{J}
= \rho J + \beta \rho^{-1} I_n
= \operatorname{BlkD}\left(\gamma I_{k_0}, \tilde{J}\right),\]
where \(\gamma
= \rho + \beta\rho^{-1},
\tilde{J}
= \gamma I_{n-k_0} - \rho M
\in \R^{(n-k_0) \times (n-k_0)}.
\)
Since $M = C_K^\top (C_K C_K^\top)^{-1} C_K$ is an orthogonal projector and $\gamma > \rho$, the block $\tilde{J}$ is strictly positive definite. Moreover, using \(M^2=M\) and \(\gamma-\rho=\beta\rho^{-1}\), we obtain
\[
\begin{aligned}
\bar{J}^{-1}
= \operatorname{BlkD}\left(\gamma^{-1} I_{k_0}, \tilde{J}^{-1}\right),\qquad
\tilde{J}^{-1}
= \frac{1}{\gamma} I_{n-k_0}
+ \frac{\rho^2}{\gamma \beta} M.
\end{aligned}
\]

The following block structure of \(M\) yields an \(O(n)\)-time application of this inverse.
\begin{proposition}\label{prop:hinge-explicit-inv}
Decompose the active set \(K\subseteq\llbracket 1,n-k_0\rrbracket\) into disjoint contiguous sequences \(K_{(q)}=\llbracket s_q,t_q\rrbracket\) of length \(m_q = t_q-s_q+1\) for \(q\in\llbracket 1,r\rrbracket\). If \(K=\emptyset\), we use the convention \(r=0\) and \(M=0\). Define \(\bar{K}_{(q)} = \llbracket s_q-1,t_q\rrbracket\) if \(s_q \ge 2\), and \(\bar{K}_{(q)} = \llbracket 1,t_q\rrbracket\) if \(s_q = 1\). Then, the projection matrix $M = C_K^\top(C_KC_K^\top)^{-1}C_K$, where \(C_K\) is the row submatrix of \(C^{n-k_0}\) indexed by \(K\), is block-diagonal with blocks indexed by \(\bar{K}_{(q)}\); all entries outside \(\cup_{q=1}^r\bar{K}_{(q)}\) are zero, and the nonzero blocks are given explicitly by
\[
M_{\bar{K}_{(q)},\,\bar{K}_{(q)}} =
\begin{cases}
I_{m_q}, & s_q = 1,\\[1mm]
I_{m_q+1} - \frac{1}{m_q+1}\mathbf{1}\mathbf{1}^\top, & s_q \ge 2.
\end{cases}
\]
\end{proposition}
\begin{proof}[Proof Sketch]
As the full derivation relies on straightforward algebraic verifications, we emphasize the structural intuitions.
If $s_q = 1$, the left boundary truncation yields an invertible lower-bidiagonal square submatrix $C_{K_{(q)}, \bar{K}_{(q)}}$, whose orthogonal projection trivially simplifies to the full-rank identity matrix $I_{m_q}$. Conversely, if $s_q \ge 2$, the local submatrix possesses a one-dimensional null space spanned by the all-ones vector $\mathbf{1}$. In this case, the orthogonal projection results in the projection matrix $I_{m_q+1} - \frac{1}{m_q+1}\mathbf{1}\mathbf{1}^\top$.
\end{proof}
{Consequently, applying $\bar{J}^{-1}$ to any vector $\bv$ costs $O(n)$ operations.}

\subsubsection{\texorpdfstring{SC$^1$ Loss}{SC1 Loss}}

For the SC$^1$ loss, let $\bx=P(\rho\bu+\bz)$, and choose $Q \in \mathcal{Q}(\bx)$ and $A \in \mathcal{A}(\bx)$. Substituting the generalized Jacobian $J = (I_n+QA)^{-1}Q$ from \eqref{eq:sc1-jac} into the shifted matrix \(\bar{J} = \rho J + \beta\rho^{-1}I_n\) yields
\[
\bar{J} = \rho(I_n+QA)^{-1}Q + \frac{\beta}{\rho}I_n = \rho(I_n+QA)^{-1} \left(Q + \frac{\beta}{\rho^2}(I_n+QA)\right),
\]
where $Q = I_n - B_K^\top (B_K B_K^\top)^{-1} B_K$ for some $K \in \mathcal{B}(\bx)$.

Defining the linear operator $\mathcal{R}(\Theta) \triangleq  B_K^\top(B_K\Theta B_K^\top)^{-1}B_K$ for any diagonal $\Theta \succ 0$, and using \(Q^2=Q\), we expand $\bar{J}^{-1}$ as
\begin{equation}\label{eq:Jinverse}
\bar{J}^{-1} = \frac{1}{\rho}\left[ \left(1+\frac{\beta}{\rho^2}\right)I_n + \frac{\beta}{\rho^2}A - \mathcal{R}(I_n)\left(I_n+\frac{\beta}{\rho^2}A\right) \right]^{-1} (I_n + A - \mathcal{R}(I_n)A).
\end{equation}
Applying the SMW formula to invert the leading matrix gives
\begin{equation}\label{eq:G-R}
\left(G - \mathcal{R}(I_n) (G - \beta\rho^{-2}I_n)\right)^{-1} = G^{-1} + G^{-1}\mathcal{R}(I_n - H)H,
\end{equation}
where the diagonal matrices $G = (1+\beta\rho^{-2})I_n + \beta\rho^{-2}A$ and $H = (G - \beta\rho^{-2}I_n)G^{-1}$ are strictly positive definite. Indeed, since \(G-\beta\rho^{-2}I_n=HG\), the leading matrix equals \((I_n-\mathcal R(I_n)H)G\), and the SMW formula gives \((I_n-\mathcal R(I_n)H)^{-1}=I_n+\mathcal R(I_n-H)H\). Using $A\succeq0$ one can verify that each diagonal entry of $H$ lies in $(0,1)$. Hence $I_n - H \succ 0$, so \(\mathcal R(I_n-H)\) is well defined.

To apply $\bar{J}^{-1}$ in $O(n)$ operations, we exploit the block-diagonal structure of $\mathcal{R}(\cdot)$ characterized below.

\begin{proposition}\label{prop:smooth-explicit-inv}
Let \(\Theta = \operatorname{diag}(\theta_1, \dots, \theta_n) \succ 0\). Decompose the active set \(K \subseteq \llbracket 1,n-1\rrbracket\) into disjoint maximal contiguous sequences \(K_{(q)} = \llbracket s_q,t_q\rrbracket\) for \(q\in\llbracket 1,r\rrbracket\). If \(K=\emptyset\), we use the convention \(r=0\) and \(\mathcal R(\Theta)=0\). Define the extended index sets \(\bar{K}_{(q)} = \llbracket s_q,t_q+1\rrbracket\). Then \(\mathcal{R}(\Theta)\) vanishes outside \(\cup_{q=1}^r\bar{K}_{(q)}\) and is block-diagonal over the sets \(\bar{K}_{(q)}\), with blocks
\[
\mathcal{R}(\Theta)_{\bar{K}_{(q)},\, \bar{K}_{(q)}} = D_{(q)} - \frac{1}{S_{(q)}} d_{(q)} d_{(q)}^\top,
\]
where \(D_{(q)} = \operatorname{diag}\!\big((\theta_i^{-1})_{i \in \bar{K}_{(q)}}\big)\), \(d_{(q)} = (\theta_i^{-1})_{i \in \bar{K}_{(q)}}\), and \(S_{(q)} = \sum_{k \in \bar{K}_{(q)}} \theta_k^{-1}\).
\end{proposition}

\begin{proof}[Proof Sketch]
Since $B_K$ decouples over the disjoint contiguous sequences $K_{(q)}$, the operator $\mathcal{R}(\Theta)$ is block-diagonal. For each local block, the null space of \(B_{K_{(q)},\,\bar{K}_{(q)}}\) is spanned by $\mathbf{1}$, and a direct local calculation gives the weighted Laplacian-type block \(D_{(q)} - S_{(q)}^{-1}d_{(q)}d_{(q)}^\top\).
\end{proof}

By \eqref{eq:Jinverse} and \eqref{eq:G-R}, applying $\bar{J}^{-1}$ to a vector involves only element-wise operations and applications of $\mathcal{R}(\cdot)$. By Proposition~\ref{prop:smooth-explicit-inv}, each application of $\mathcal{R}(\cdot)$ costs $O(n)$ operations via block-wise inner products. Hence, $\bar{J}^{-1}\bv$ can be computed in $O(n)$ operations.

\section{Numerical Experiments}

We implement the proposed method in Python 3.8. The outer iterations follow the \textnormal{\textsc{ripALM}} scheme (Algorithm~\ref{alg:rALM}), while each inner subproblem \eqref{problem:inexact-u-subproblem} is solved by the \textnormal{\textsc{Ssn}} method (Algorithm~\ref{alg:ssn-dual}). All experiments are executed on a Linux server equipped with an AMD EPYC 7402 CPU (2.8 GHz, 96 cores) and 256 GB RAM.

\subsection{Experimental Setup}
\label{sec:experimental-setup}

To evaluate the optimality of the computed solutions for both the proposed \textnormal{\textsc{ripALM}}-\textnormal{\textsc{Ssn}} method and the ADMM baseline, we measure their discrepancy against the KKT conditions of \eqref{eq:dual-problem}. 
In all reported binary-classification experiments, \(D=-\by\circ X\) and the affine offset is \(\bc=\bs{0}\), so the constraint reduces to \(D\bw-\bz=\bs{0}\).
We employ the relative primal and dual KKT residuals defined respectively as
{\small\begin{equation}
\label{eq:kkt-residual}
    \begin{aligned}
    \eta_{p} &= \max\left\{\frac{\|D^\top\bu + \bs{\xi}\|}{1+\|D^\top\bu\|+ \|\bs{\xi}\|}, \frac{\|\bu-\bs{\zeta}\|}{1+\|\bu\| + \|\bs{\zeta}\|}\right\}, \\
    \eta_{d} &= \max\bigg\{\frac{\|D\bw-\bz\|}{1+\|D\bw\| + \|\bz\|}, \frac{\|\bs{\xi}-\Pi_{B^\lambda_{\Vert \cdot\Vert_\infty}}(\bw+\bs{\xi})\|}{1+\|\bs{\xi}\| + \|\bw\|}, \frac{\|\bz - \prox_{f}(\bz+\bs{\zeta})\|}{1+\|\bs{\zeta}\| + \|\bz\|}\bigg\}.
    \end{aligned}
\end{equation}}
Here $B^\lambda_{\|\cdot\|_\infty}\triangleq \{\bs{x}\in\R^d\mid \|\bs{x}\|_\infty\le \lambda\}$, and $\Pi_C$ denotes the Euclidean projection onto a closed convex set $C$.
Since evaluating $f^*$ is intractable, we also monitor the relative objective change $\text{obj}_{\text{gap}} = |\mathcal{F}(\bw_{k+1})-\mathcal{F}(\bw_k)|/(1+|\mathcal{F}(\bw_k)|)$, where $\mathcal{F}$ is the primal objective \eqref{eq:primal-problem}. Both methods (the \textnormal{\textsc{ripALM}}-\textnormal{\textsc{Ssn}} method and the ADMM baseline in Appendix~\ref{sec:admm}) terminate when $\max\{\eta_p, \eta_d\} \leq 10^{-5}$, or when $\text{obj}_{\text{gap}} \leq 10^{-7}$ alongside $\max\{\eta_p, \eta_d\} \leq 10^{-4}$.

We evaluate three spectral functions sourced from \eqref{eq:srm_cvar}, \eqref{eq:srm_esrm}, and \eqref{eq:srm_extremile}. The parameters are fixed as $(\nu, \varrho, r) = (0.15, 0.1, 1.05)$.
Following \cite{koh2007method}, the regularization parameter \(\lambda\) is scaled by 
$
\lambda_{\rm scale}:=\max_{P\in\Pi_n}
\|a_\ell D^\top(P\boldsymbol{\sigma})\|_\infty,
$
where $\Pi_n$ denotes the set of all $n\times n$ permutation matrices, \(D=-\by\circ X\), \(a_\ell=1/2\) for the logistic loss, and \(a_\ell=1\) for the hinge and smoothed hinge losses.

For the \textnormal{\textsc{ripALM}}-\textnormal{\textsc{Ssn}} method, we initialize all variables to zero. We set $\tau = 0.05$ in Algorithm~\ref{alg:rALM}, and $\alpha=1$ and $\theta=0.7$ in Algorithm~\ref{alg:ssn-dual}.
For the hinge loss, we fix $\beta_k = \max\{50\lambda, 5\}$ and update $\rho_k = \max\{50\lambda, 20\} \cdot 2^k$.
For the SC$^1$ loss, we employ $\beta_k = \max\{20\lambda, 5\}$ and $\rho_k = \max\{20\lambda, 20\} \cdot 3^{\lfloor k/2\rfloor}$.
The ADMM baseline employs tuning strategies identical to \cite{wu2023convex}.

\subsection{Synthetic Dataset Experiment}

To investigate scalability, we follow the protocol in~\cite{liu2023dual} by setting $(n,d) = (250i,\, 5000i)$ for \(i \in \llbracket 1,9\rrbracket\). Following \cite{koh2007method}, we construct balanced binary datasets where positive and negative features are drawn from $\mathcal{N}(1,1)$ and $\mathcal{N}(-1,1)$, respectively. The feature matrices $X$ are uniformly sparsified to a $70\%$ sparsity level. The regularization parameter is fixed at $\lambda = 0.1\lambda_{\rm scale}$.
\begin{figure}[!h]
    \centering
    \includegraphics[width=\textwidth]{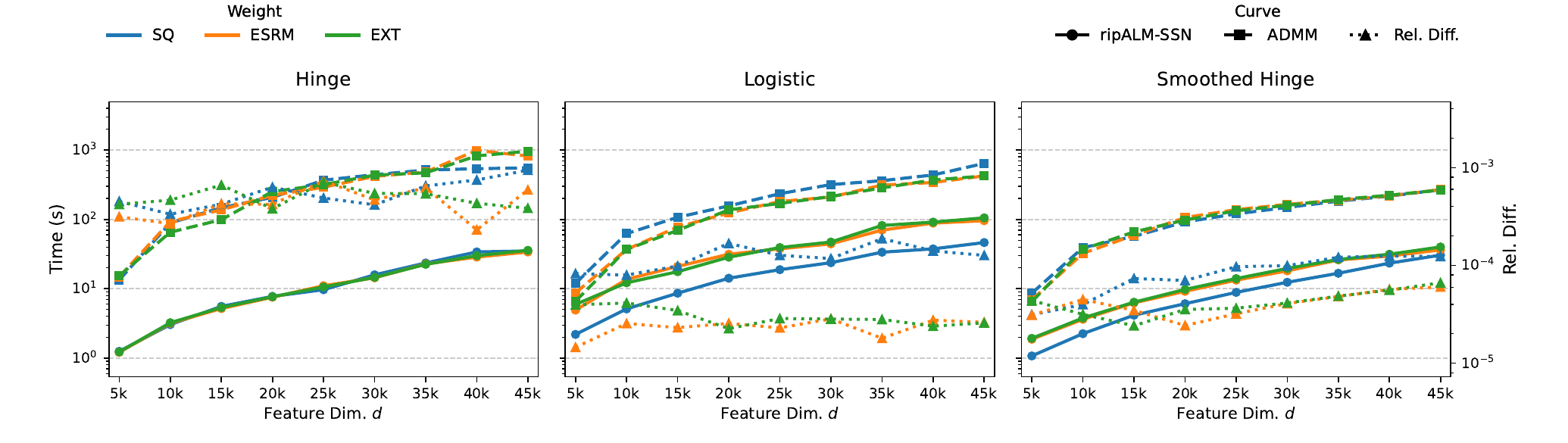}

    \caption{Average computation time and average relative objective difference across problem dimensions, computed over five independent datasets for each scale. From left to right, the panels correspond to the hinge, logistic, and smoothed hinge losses. Colors denote the weight functions Superquantile (SQ), ESRM, and Extremile (EXT), while line styles distinguish \textnormal{\textsc{ripALM}}-\textnormal{\textsc{Ssn}}, ADMM, and the Rel. Diff. \eqref{eq:obj-diff}.}
    \label{fig:syn-exp}
\end{figure}

Figure~\ref{fig:syn-exp} reports averages over five independent datasets at each scale. The metric ``Rel. Diff.'' computes the relative objective difference formulated as
\begin{equation}
    \label{eq:obj-diff}
    \text{Rel. Diff.} = \frac{ \mathcal{F}(\bw_{\text{ADMM}}) - \mathcal{F}(\bw_{\text{\textnormal{\textsc{ripALM}}-\textnormal{\textsc{Ssn}}}})}
     {1 + |\mathcal{F}(\bw_{\text{ADMM}})|}.
\end{equation}
As shown in Figure~\ref{fig:syn-exp}, the \textnormal{\textsc{ripALM}}-\textnormal{\textsc{Ssn}} method scales favorably with the problem dimension. It converges within $100$ seconds across all tested scales, whereas ADMM is several times to over one order of magnitude slower and exceeds 1000 seconds in the hardest hinge-loss cases. Moreover, the relative differences remain positive across all tested settings, indicating that the \textnormal{\textsc{ripALM}}-\textnormal{\textsc{Ssn}} method attains lower objective values than ADMM on average. These results demonstrate the efficiency, scalability, and solution quality of the proposed \textnormal{\textsc{ripALM}}-\textnormal{\textsc{Ssn}} method.

\subsection{Real-World Data Experiment}
\begin{table}[ht]
\centering
\caption{Comparison results on real datasets across different loss functions. For Time, Iter, and KKT, entries $a/b$ denote \textnormal{\textsc{ripALM}}-\textnormal{\textsc{Ssn}}/ADMM. For the \textnormal{\textsc{ripALM}}-\textnormal{\textsc{Ssn}} method, Iter denotes outer (total inner) iterations. KKT denotes \(10^6\max\{\eta_p,\eta_d\}\), with \(\eta_p,\eta_d\) in \eqref{eq:kkt-residual}. Diff reports \eqref{eq:obj-diff}, and NNZ reports the number of nonzeros in the solution obtained by the \textnormal{\textsc{ripALM}}-\textnormal{\textsc{Ssn}} method.}
\label{tab:real-data-comparison}
{
\tiny
\setlength{\tabcolsep}{0pt}
\renewcommand{\arraystretch}{1.02}
\begin{tabular}{l|c@{\hspace{0.8pt}}c@{\hspace{0.8pt}}c@{\hspace{1.6pt}}|@{\hspace{1.6pt}}c@{\hspace{0.8pt}}c@{\hspace{0.8pt}}c@{\hspace{1.6pt}}|@{\hspace{1.6pt}}c@{\hspace{0.8pt}}c@{\hspace{0.8pt}}c}
\toprule[1.5pt]
\textbf{Data} & \multicolumn{3}{c}{\shortstack{\textbf{Dexter} (300,20000)}} & \multicolumn{3}{c}{\shortstack{\textbf{LiverC} (76,36547)}} & \multicolumn{3}{c}{\shortstack{\textbf{BreastC} (151,54675)}} \\
\midrule[1.2pt]
$\tfrac{\lambda}{\lambda_{\rm scale}}$ & 0.20 & 0.07 & 0.03 & 0.20 & 0.07 & 0.03 & 0.20 & 0.07 & 0.03 \\
\midrule[1.2pt]
\addlinespace[0.35em]
\multicolumn{1}{l}{} & \multicolumn{9}{l}{\makebox[20em][l]{\textbf{Logistic Loss:}}\textbf{ESRM}} \\
\midrule[0.5pt]
~ Time & 3.18/10.7 & 4.51/28.1 & 5.82/21.8 & 2.59/53.8 & 3.24/60.1 & 5.02/50.7 & 8.15/109.2 & 8.73/88.1 & 9.85/139.8 \\
~ Iter & 7(82)/221 & 11(90)/573 & 16(114)/617 & 9(63)/1365 & 15(68)/1233 & 21(90)/1145 & 10(108)/1189 & 15(98)/969 & 18(106)/1453 \\
~ Diff & 3.3e-05 & 1.4e-05 & 4.6e-05 & 2.7e-04 & 2.7e-04 & 3.5e-04 & 3.6e-04 & 2.7e-04 & 1.1e-04 \\
~ KKT & 9.2/6.6 & 5.2/7.7 & 6.1/44 & 3.4/5.5 & 7.4/49 & 8.1/100 & 6.7/9.0 & 4.8/8.6 & 7.2/9.2 \\
~ NNZ & 42 & 148 & 170 & 26 & 44 & 49 & 46 & 93 & 108 \\
\midrule[0.75pt]
\multicolumn{1}{l}{} & \multicolumn{9}{l}{\makebox[20em][l]{\textbf{Logistic Loss:}}\textbf{Extremile}} \\
\midrule[0.5pt]
~ Time & 3.80/10.3 & 5.00/26.4 & 5.75/32.8 & 2.52/37.8 & 3.23/52.3 & 5.35/86.3 & 8.49/98.7 & 8.24/96.5 & 12.11/130.9 \\
~ Iter & 7(95)/221 & 11(103)/573 & 16(119)/705 & 9(65)/1013 & 15(77)/1277 & 21(98)/2289 & 10(108)/1189 & 14(100)/969 & 18(107)/1453 \\
~ Diff & 3.1e-05 & 1.4e-05 & 2.6e-05 & 1.2e-03 & 2.0e-04 & 9.6e-05 & 3.9e-04 & 2.7e-04 & 1.0e-04 \\
~ KKT & 9.0/5.9 & 4.6/6.4 & 5.2/23 & 3.1/77 & 5.9/41 & 7.2/9.7 & 6.1/9.7 & 8.5/8.8 & 6.5/9.5 \\
~ NNZ & 42 & 145 & 169 & 26 & 43 & 49 & 47 & 93 & 106 \\
\midrule[0.75pt]
\multicolumn{1}{l}{} & \multicolumn{9}{l}{\makebox[20em][l]{\textbf{Logistic Loss:}}\textbf{Superquantile}} \\
\midrule[0.5pt]
~ Time & 1.74/14.0 & 3.36/19.8 & 4.96/30.6 & 1.43/68.1 & 2.64/52.0 & 3.68/97.8 & 4.37/245.4 & 5.56/125.3 & 8.07/137.9 \\
~ Iter & 8(38)/353 & 10(56)/485 & 13(90)/661 & 8(30)/1849 & 12(50)/1277 & 19(68)/2509 & 9(44)/3125 & 12(62)/1233 & 17(85)/1673 \\
~ Diff & 7.6e-05 & 2.2e-05 & 1.2e-05 & 2.8e-04 & 2.3e-04 & 5.8e-05 & 6.4e-04 & 2.6e-04 & 2.3e-04 \\
~ KKT & 3.0/9.0 & 8.7/7.0 & 6.0/6.7 & 2.7/5.6 & 9.8/6.0 & 6.3/14 & 1.9/9.6 & 8.7/7.7 & 6.5/8.6 \\
~ NNZ & 36 & 140 & 169 & 17 & 43 & 48 & 36 & 100 & 115 \\
\midrule[1.2pt]
$\tfrac{\lambda}{\lambda_{\rm scale}}$ & 0.10 & 0.05 & 0.01 & 0.10 & 0.05 & 0.01 & 0.10 & 0.05 & 0.01 \\
\midrule[1.2pt]
\addlinespace[0.35em]
\multicolumn{1}{l}{} & \multicolumn{9}{l}{\makebox[20em][l]{\textbf{Smoothed Hinge Loss:}}\textbf{ESRM}} \\
\midrule[0.5pt]
~ Time & 1.75/16.8 & 2.54/16.7 & 4.70/78.7 & 1.95/98.2 & 2.70/95.7 & 4.74/96.8 & 4.41/138.9 & 6.31/179.7 & 9.78/353.7 \\
~ Iter & 6(40)/353 & 8(55)/353 & 16(97)/1673 & 8(39)/2553 & 9(48)/2509 & 20(79)/2421 & 8(53)/1761 & 11(66)/1937 & 17(101)/3829 \\
~ Diff & 1.3e-04 & 1.2e-04 & 3.5e-05 & 1.2e-03 & 6.0e-04 & 6.2e-04 & 1.0e-03 & 6.0e-04 & 3.8e-04 \\
~ KKT & 9.0/5.8 & 6.4/8.9 & 8.3/15 & 6.0/9.0 & 8.1/6.7 & 9.6/58 & 6.7/9.5 & 4.0/9.3 & 6.7/21 \\
~ NNZ & 147 & 202 & 315 & 51 & 59 & 64 & 99 & 126 & 139 \\
\midrule[0.75pt]
\multicolumn{1}{l}{} & \multicolumn{9}{l}{\makebox[20em][l]{\textbf{Smoothed Hinge Loss:}}\textbf{Extremile}} \\
\midrule[0.5pt]
~ Time & 1.80/17.8 & 2.69/17.9 & 4.67/53.1 & 2.02/99.5 & 2.34/98.6 & 5.55/105.7 & 4.42/172.0 & 6.07/180.2 & 8.65/241.9 \\
~ Iter & 7(44)/353 & 8(54)/353 & 16(95)/1101 & 8(39)/2553 & 9(46)/2465 & 20(75)/2861 & 8(54)/1761 & 11(66)/1937 & 17(96)/2641 \\
~ Diff & 1.3e-04 & 1.2e-04 & 1.9e-04 & 1.2e-03 & 6.4e-04 & 2.3e-04 & 9.9e-04 & 6.4e-04 & 5.3e-04 \\
~ KKT & 1.2/5.7 & 5.6/8.8 & 7.9/54 & 5.2/7.6 & 6.8/9.5 & 8.2/25 & 6.2/9.2 & 3.9/9.6 & 6.5/89 \\
~ NNZ & 146 & 201 & 315 & 51 & 58 & 64 & 98 & 126 & 139 \\
\midrule[0.75pt]
\multicolumn{1}{l}{} & \multicolumn{9}{l}{\makebox[20em][l]{\textbf{Smoothed Hinge Loss:}}\textbf{Superquantile}} \\
\midrule[0.5pt]
~ Time & 1.68/17.3 & 1.38/19.8 & 4.11/58.2 & 1.56/98.4 & 1.88/107.0 & 3.88/79.6 & 3.75/206.4 & 5.37/220.2 & 9.12/270.7 \\
~ Iter & 7(31)/353 & 7(32)/397 & 14(67)/1453 & 7(29)/2597 & 8(27)/2729 & 16(55)/1981 & 7(38)/2245 & 9(49)/2333 & 17(86)/2861 \\
~ Diff & 3.0e-04 & 9.3e-05 & 4.0e-05 & 1.8e-03 & 8.8e-04 & 1.1e-03 & 1.4e-03 & 7.2e-04 & 3.2e-04 \\
~ KKT & 4.4/8.6 & 5.7/6.9 & 8.9/16 & 4.0/8.5 & 5.1/8.7 & 10/76 & 6.7/9.3 & 7.8/9.5 & 7.3/29 \\
~ NNZ & 116 & 192 & 313 & 46 & 54 & 62 & 96 & 122 & 140 \\
\midrule[0.75pt]
\addlinespace[0.35em]
\multicolumn{1}{l}{} & \multicolumn{9}{l}{\makebox[20em][l]{\textbf{Hinge Loss:}}\textbf{ESRM}} \\
\midrule[0.5pt]
~ Time & 3.41/253.2 & 3.70/66.3 & 4.74/123.7 & 3.44/194.9 & 4.53/423.4 & 8.21/225.0 & 4.70/425.9 & 8.58/414.5 & 9.81/281.6 \\
~ Iter & 7(82)/5149 & 9(74)/1409 & 16(98)/2597 & 8(49)/4885 & 10(50)/10561 & 20(93)/5501 & 8(63)/5765 & 11(80)/6117 & 15(102)/3125 \\
~ Diff & 2.6e-04 & 1.1e-03 & 3.8e-04 & 5.4e-03 & 1.3e-03 & 1.1e-03 & 2.8e-03 & 2.2e-03 & 2.2e-03 \\
~ KKT & 3.6/9.7 & 1.9/63 & 6.6/89 & 1.7/46 & 1.2/30 & 9.5/68 & 6.8/29 & 4.8/35 & 9.3/88 \\
~ NNZ & 310 & 311 & 312 & 63 & 63 & 63 & 141 & 140 & 142 \\
\midrule[0.75pt]
\multicolumn{1}{l}{} & \multicolumn{9}{l}{\makebox[20em][l]{\textbf{Hinge Loss:}}\textbf{Extremile}} \\
\midrule[0.5pt]
~ Time & 3.67/196.2 & 3.20/123.9 & 5.20/108.5 & 4.25/301.6 & 3.39/463.2 & 11.25/175.0 & 7.68/656.2 & 10.59/319.0 & 9.00/359.9 \\
~ Iter & 7(81)/4049 & 9(72)/2597 & 16(98)/2377 & 8(49)/8273 & 9(51)/13641 & 20(95)/4577 & 8(70)/6865 & 11(83)/3433 & 15(94)/3697 \\
~ Diff & 4.1e-04 & 4.1e-04 & 4.0e-04 & 2.9e-03 & 1.0e-03 & 1.7e-03 & 2.2e-03 & 4.0e-03 & 2.1e-03 \\
~ KKT & 3.6/15 & 1.5/30 & 6.1/99 & 1.5/21 & 9.9/14 & 8.8/66 & 7.6/26 & 4.0/68 & 9.4/74 \\
~ NNZ & 310 & 311 & 311 & 63 & 63 & 63 & 141 & 140 & 142 \\
\midrule[0.75pt]
\multicolumn{1}{l}{} & \multicolumn{9}{l}{\makebox[20em][l]{\textbf{Hinge Loss:}}\textbf{Superquantile}} \\
\midrule[0.5pt]
~ Time & 4.59/195.9 & 3.12/207.8 & 4.64/134.0 & 3.88/155.5 & 4.72/460.5 & 5.92/539.9 & 5.48/490.1 & 7.30/529.0 & 9.81/438.4 \\
~ Iter & 9(103)/4269 & 7(64)/4181 & 13(98)/2773 & 7(47)/4181 & 9(51)/11309 & 15(68)/13861 & 8(68)/5589 & 9(74)/5985 & 15(96)/4577 \\
~ Diff & 5.9e-04 & 2.3e-04 & 4.3e-04 & 7.8e-03 & 2.0e-03 & 4.2e-04 & 3.9e-03 & 2.3e-03 & 1.9e-03 \\
~ KKT & 7.4/19 & 6.1/15 & 8.1/69 & 7.5/48 & 6.9/19 & 0.6/17 & 2.3/30 & 9.3/34 & 6.4/81 \\
~ NNZ & 290 & 309 & 312 & 63 & 63 & 63 & 140 & 142 & 140 \\
\bottomrule[1.5pt]
\end{tabular}
}
\end{table}

To evaluate performance in regimes with high dimensions and small sample sizes ($d \gg n$), we test three real datasets.
These are Dexter (UCI\footnote{\url{https://archive.ics.uci.edu/datasets}}), LiverC and BreastC (CuMiDa \cite{cumida:2019}\footnote{\url{https://sbcb.inf.ufrgs.br/cumida}}).

Table~\ref{tab:real-data-comparison} summarizes the results across three $\lambda$ values for each loss function.
Empirically, the \textnormal{\textsc{ripALM}}-\textnormal{\textsc{Ssn}} method demonstrates a clear computational advantage, converging within $10$ seconds in almost all instances, whereas ADMM's runtime degrades severely on datasets with tens of thousands of features. Moreover, the strictly positive ``Diff'' values confirm that the \textnormal{\textsc{ripALM}}-\textnormal{\textsc{Ssn}} method consistently achieves slightly lower (better) objective values than ADMM, while maintaining comparable or superior KKT residuals. Finally, the ``NNZ'' metric verifies that the proposed method reliably extracts highly sparse solutions; the mild sparsity variations under the hinge loss align with established $\ell_1$-regularization traits \cite{moore2011l1}. Collectively, these outcomes indicate favorable computational performance on high-dimensional instances.

\subsection{Adaptive Sieving Acceleration for Solution Paths}

For the solution-path experiments, we embed the \textnormal{\textsc{ripALM}}-\textnormal{\textsc{Ssn}} method in the adaptive sieving (AS) framework of~\cite{wu2023convex}. For each regularization parameter, AS solves the restricted problem on the current active feature set, lifts the restricted solution to the full space, and evaluates the full-space proximal KKT residual. Variables whose residuals exceed the screening tolerance are added to the active set, and this process repeats until the screening test is satisfied; the final solution is then checked by the stopping criteria in Section~\ref{sec:experimental-setup}. We set the screening tolerance to \(\varepsilon=10^{-6}\) and use the same termination criteria for each restricted solve. We compare \texttt{AS} with two full-space variants of \textnormal{\textsc{ripALM}}-\textnormal{\textsc{Ssn}}: \texttt{Warm}, initialized at the solution from the previous \(\lambda\), and \texttt{Cold}, initialized at zero. For Syn1 and Syn2, we average over three independent random seeds. All entries in Table~\ref{tab:as-comparison} report average wall-clock time per \(\lambda\) on the same regularization path for each fixed data set, loss function, and spectral weight.

As shown in Table~\ref{tab:as-comparison}, the AS-equipped \textnormal{\textsc{ripALM}}-\textnormal{\textsc{Ssn}} method
reduces the average solution-path time relative to both \texttt{Warm}
and \texttt{Cold} on all tested instances. The largest observed speedup
over \texttt{Cold} is $68\times$. Similar reductions are obtained across
the three spectral weight choices and the three loss functions, although
the speedup depends on the data set and the loss function. These results
show that the proposed method can exploit the active-set sparsity along
the regularization path.

\section{Conclusion}
We developed the \textnormal{\textsc{ripALM}}-\textnormal{\textsc{Ssn}} method for sparse spectral risk optimization. The method combines an enhanced PAVA algorithm for evaluating the SRM-based proximal mapping with an explicit characterization of the \textit{HS--Jacobian} of the SRM-based proximal mapping, leading to computable Newton systems for the \textnormal{\textsc{Ssn}} method. The numerical results show lower running times than the tested ADMM baseline and indicate that the method can be combined with adaptive sieving for solution-path computations.
Future work will investigate extensions beyond convex monotone losses, including nonconvex losses and nonlinear models.

\begin{table}[tbp]
\centering

\caption{Adaptive sieving for solution paths. Each entry reports average seconds per parameter value as \texttt{AS}/\texttt{Warm}/\texttt{Cold}; the parenthesized factor attached to \texttt{AS} is the speedup relative to \texttt{Cold}.}
\label{tab:as-comparison}
\setlength{\tabcolsep}{4.2pt}
\renewcommand{\arraystretch}{0.86}\footnotesize
\begin{tabular}{@{}l@{\hspace{0.7em}}l@{\:/\:}l@{\:/\:}l@{\hspace{1.2em}}l@{\:/\:}l@{\:/\:}l@{\hspace{1.2em}}l@{\:/\:}l@{\:/\:}l@{}}
\toprule[1.5pt]
\textbf{Dataset} & \multicolumn{3}{l}{\textbf{ESRM}} & \multicolumn{3}{l}{\textbf{Extremile}} & \multicolumn{3}{l}{\textbf{Superquantile}} \\
\midrule[1.0pt]
\multicolumn{10}{c}{\textbf{Logistic Loss}} \\
\midrule[0.5pt]
LiverC (76,36547) & 0.11(29x) & 2.40 & 3.32 & 0.11(37x) & 2.70 & 4.15 & 0.08(37x) & 2.09 & 2.97 \\
BreastC (151,54675) & 2.84(4x) & 7.53 & 12.37 & 3.58(3x) & 6.32 & 10.31 & 0.66(10x) & 4.29 & 6.41 \\
Syn1 (200,25000) & 1.56(12x) & 11.97 & 18.50 & 1.93(10x) & 11.93 & 19.74 & 0.50(15x) & 4.14 & 7.61 \\
Syn2 (500,50000) & 17.33(4x) & 40.11 & 73.99 & 19.98(4x) & 41.98 & 78.23 & 5.75(5x) & 14.37 & 26.78 \\
\midrule[1.0pt]
\multicolumn{10}{c}{\textbf{Hinge Loss}} \\
\midrule[0.5pt]
LiverC (76,36547) & 0.08(68x) & 0.91 & 5.71 & 0.09(66x) & 0.87 & 5.94 & 0.08(64x) & 1.04 & 5.39 \\
BreastC (151,54675) & 0.25(54x) & 1.80 & 13.44 & 0.23(58x) & 1.90 & 13.36 & 0.28(45x) & 1.99 & 12.53 \\
Syn1 (200,25000) & 0.46(10x) & 1.40 & 4.69 & 0.44(11x) & 1.40 & 4.68 & 0.43(11x) & 1.51 & 4.56 \\
Syn2 (500,50000) & 1.77(9x) & 4.66 & 16.39 & 1.72(9x) & 4.76 & 16.26 & 1.71(9x) & 5.04 & 15.82 \\
\midrule[1.0pt]
\multicolumn{10}{c}{\textbf{Smoothed Hinge Loss}} \\
\midrule[0.5pt]
LiverC (76,36547) & 0.09(30x) & 1.41 & 2.81 & 0.08(34x) & 1.41 & 2.78 & 0.07(31x) & 1.18 & 2.08 \\
BreastC (151,54675) & 0.51(19x) & 3.83 & 9.52 & 0.73(14x) & 4.43 & 9.95 & 0.14(44x) & 2.89 & 6.40 \\
Syn1 (200,25000) & 0.31(22x) & 5.10 & 6.78 & 0.31(22x) & 5.16 & 6.88 & 0.20(25x) & 2.56 & 4.89 \\
Syn2 (500,50000) & 1.25(18x) & 22.59 & 22.86 & 1.28(17x) & 22.78 & 21.89 & 0.81(19x) & 8.23 & 15.15 \\
\bottomrule[1.5pt]
\end{tabular}

\end{table}

\appendix
\section{ADMM Framework}
\label{sec:admm}
The ADMM baseline applies ADMM to the two-auxiliary-variable dual formulation in \eqref{eq:dual-problem}.
Moreau's decomposition and the quadratic \(\bu\)-subproblem yield the following explicit ADMM updates:
\begin{equation}\nonumber
\begin{aligned}
    \bs{\xi}_{k+1} & = \Pi_{B^\lambda_{\Vert \cdot\Vert_\infty}}\left(\frac{\bw_k}{\rho_k} -D^\top \bu_k\right), \\[1ex]
    \bs{\zeta}_{k+1} & = \bu_k+ \frac{\bz_k}{\rho_k} - \frac{1}{\rho_k}\prox_{\rho_k f}\left(\rho_k \bu_k+ \bz_k\right), \\[1ex]
    \bu_{k+1} & = (I+DD^\top)^{-1} \left( D\left(\frac{\bw_{k}}{\rho_k} - \bs{\xi}_{k+1}\right) + \bs{\zeta}_{k+1} - \frac{\bz_{k}}{\rho_k} + \frac{\bc}{\rho_k} \right).
\end{aligned}
\end{equation}
The dual variables are updated by
\vspace{-0.5em}
\[
    \bw_{k+1} = \bw_{k} - \rho_k(D^\top\bu_{k+1} + \bs{\xi}_{k+1}), \qquad
    \bz_{k+1} = \bz_k + \rho_k(\bu_{k+1} - \bs{\zeta}_{k+1}).
\]
The ADMM hyperparameters follow the experimental setting in \cite{wu2023convex}.

\bibliographystyle{siamplain}
\bibliography{ref}

@article{robinson1981some,
  title={Some continuity properties of polyhedral multifunctions},
  author={Robinson, Stephen M},
  journal={Math. Programming Stud.},
  volume={14},
  pages={206--214},
  year={1981},
  publisher={Springer}
}

@article{han1997newton,
  title={{Newton} and quasi-{Newton} methods for normal maps with polyhedral sets},
  author={Han, Jiye and Sun, Defeng},
  journal={J. Optim. Theory Appl.},
  volume={94},
  number={3},
  pages={659--676},
  year={1997},
  publisher={Springer}
}

@book{beck2017first,
  title={First-order methods in optimization},
  author={Beck, Amir},
  year={2017},
  publisher={SIAM}
}

@book{rockafellar2009variational,
  title={Variational analysis},
  author={Rockafellar, R Tyrrell and Wets, Roger J-B},
  year={2009},
  publisher={Springer Science \& Business Media}
}

@article{Qi1993semismooth,
author = {Qi, Liqun and Sun, Jie},
title = {A nonsmooth version of {Newton}'s method},
year = {1993},
issue_date = {January   1993},
publisher = {Springer-Verlag},
address = {Berlin, Heidelberg},
volume = {58},
number = {1--3},
issn = {0025-5610},
journal = {Math. Program.},
month = jan,
pages = {353--367},
numpages = {15}
}

@book{clarke1990nonsmooth,
author = {Clarke, Frank H.},
title = {Optimization and Nonsmooth Analysis},
publisher = {Society for Industrial and Applied Mathematics},
year = {1990}
}

@article{gowda2004inverse,
  title={Inverse and implicit function theorems for {H}-differentiable and semismooth functions},
  author={Gowda, M Seetharama},
  journal={Optim. Methods Softw.},
  volume={19},
  number={5},
  pages={443--461},
  year={2004},
  publisher={Taylor \& Francis}
}

@article{li2018efficiently,
  title={On efficiently solving the subproblems of a level-set method for fused lasso problems},
  author={Li, Xudong and Sun, Defeng and Toh, Kim-Chuan},
  journal={SIAM J. Optim.},
  volume={28},
  number={2},
  pages={1842--1866},
  year={2018},
  publisher={SIAM}
}

@book{facchinei2003finite,
  title={Finite-dimensional variational inequalities and complementarity problems},
  author={Facchinei, Francisco and Pang, Jong-Shi},
  year={2003},
  publisher={Springer}
}

@article{chow2015risk,
  title={Risk-sensitive and robust decision-making: a {CVaR} optimization approach},
  author={Chow, Yinlam and Tamar, Aviv and Mannor, Shie and Pavone, Marco},
  journal={Adv. Neural Inf. Process. Syst.},
  volume={28},
  year={2015}
}

@inproceedings{dwork2012fairness,
  title={Fairness through awareness},
  author={Dwork, Cynthia and Hardt, Moritz and Pitassi, Toniann and Reingold, Omer and Zemel, Richard},
  booktitle={Proceedings of the 3rd Innovations in Theoretical Computer Science Conference},
  pages={214--226},
  year={2012}
}

@misc{acerbi2002portfolio,
  title={Portfolio optimization with spectral measures of risk},
  author={Acerbi, Carlo and Simonetti, Prospero},
  howpublished={arXiv preprint cond-mat/0203607},
  year={2002}
}

@article{acerbi2002spectral,
  title={Spectral measures of risk: A coherent representation of subjective risk aversion},
  author={Acerbi, Carlo},
  journal={J. Bank. Finance},
  volume={26},
  number={7},
  pages={1505--1518},
  year={2002},
  publisher={Elsevier}
}

@inproceedings{maurer2021robust,
  title={Robust unsupervised learning via {L}-statistic minimization},
  author={Maurer, Andreas and Parletta, Daniela Angela and Paudice, Andrea and Pontil, Massimiliano},
  booktitle={International Conference on Machine Learning},
  pages={7524--7533},
  year={2021},
  organization={PMLR}
}

@article{daouia2019extremiles,
  title={Extremiles: A new perspective on asymmetric least squares},
  author={Daouia, Abdelaati and Gijbels, Ir{\`e}ne and Stupfler, Gilles},
  journal={J. Amer. Statist. Assoc.},
  volume={114},
  number={527},
  pages={1366--1381},
  year={2019},
  publisher={Taylor \& Francis}
}

@article{rockafellar2000optimization,
  title={Optimization of conditional value-at-risk},
  author={Rockafellar, R Tyrrell and Uryasev, Stanislav},
  journal={J. Risk},
  volume={2},
  pages={21--42},
  year={2000}
}

@article{wu2023convex,
  title={Convex and nonconvex risk-based linear regression at scale},
  author={Wu, Can and Cui, Ying and Li, Donghui and Sun, Defeng},
  journal={INFORMS J. Comput.},
  volume={35},
  number={4},
  pages={797--816},
  year={2023},
  publisher={INFORMS}
}

@article{cui2025decision,
  title={Decision making under cumulative prospect theory: An alternating direction method of multipliers},
  author={Cui, Xiangyu and Jiang, Rujun and Shi, Yun and Xiao, Rufeng and Yan, Yifan},
  journal={INFORMS J. Comput.},
  volume={37},
  number={4},
  pages={856--873},
  year={2025},
  publisher={INFORMS}
}

@article{xiao2023unified,
  title={A unified framework for rank-based loss minimization},
  author={Xiao, Rufeng and Ge, Yuze and Jiang, Rujun and Yan, Yifan},
  journal={Adv. Neural Inf. Process. Syst.},
  volume={36},
  pages={51302--51326},
  year={2023}
}

@article{best2000minimizing,
  title={Minimizing separable convex functions subject to simple chain constraints},
  author={Best, Michael J and Chakravarti, Nilotpal and Ubhaya, Vasant A},
  journal={SIAM J. Optim.},
  volume={10},
  number={3},
  pages={658--672},
  year={2000},
  publisher={SIAM}
}

@inproceedings{mehta2023stochastic,
  title={Stochastic optimization for spectral risk measures},
  author={Mehta, Ronak and Roulet, Vincent and Pillutla, Krishna and Liu, Lang and Harchaoui, Zaid},
  booktitle={International Conference on Artificial Intelligence and Statistics},
  pages={10112--10159},
  year={2023},
  organization={PMLR}
}

@inproceedings{
ge2025sorel,
title={{SOREL}: A Stochastic Algorithm for Spectral Risks Minimization},
author={Yuze Ge and Rujun Jiang},
booktitle={The Thirteenth International Conference on Learning Representations},
year={2025}
}

@misc{zhu2025ripalm,
      title={{ripALM}: A relative-type inexact proximal augmented {Lagrangian} method for linearly constrained convex optimization}, 
      author={Jiayi Zhu and Ling Liang and Lei Yang and Kim-Chuan Toh},
      howpublished={arXiv preprint arXiv:2411.13267},
      year={2025},
}

@article{li2020asymptotically,
  title={An asymptotically superlinearly convergent semismooth {Newton} augmented {Lagrangian} method for linear programming},
  author={Li, Xudong and Sun, Defeng and Toh, Kim-Chuan},
  journal={SIAM J. Optim.},
  volume={30},
  number={3},
  pages={2410--2440},
  year={2020},
  publisher={SIAM}
}

@article{yang2024corrected,
  title={A corrected inexact proximal augmented {Lagrangian} method with a relative error criterion for a class of group-quadratic regularized optimal transport problems},
  author={Yang, Lei and Liang, Ling and Chu, Hong TM and Toh, Kim-Chuan},
  journal={J. Sci. Comput.},
  volume={99},
  number={3},
  pages={79},
  year={2024},
  publisher={Springer}
}

@book{nocedal2006numerical,
  title={Numerical optimization},
  author={Nocedal, Jorge and Wright, Stephen J},
  year={2006},
  publisher={Springer}
}

@article{jiang1995local,
  title={Local uniqueness and convergence of iterative methods for nonsmooth variational inequalities},
  author={Jiang, Houyuan Y and Qi, LQ},
  journal={J. Math. Anal. Appl.},
  volume={196},
  number={1},
  pages={314--331},
  year={1995},
  publisher={Elsevier}
}

@article{zhao2010newton,
  title={A {Newton-CG} augmented {Lagrangian} method for semidefinite programming},
  author={Zhao, Xin-Yuan and Sun, Defeng and Toh, Kim-Chuan},
  journal={SIAM J. Optim.},
  volume={20},
  number={4},
  pages={1737--1765},
  year={2010},
  publisher={SIAM}
}

@article{klatte2018approximations,
  title={Approximations and generalized {Newton} methods},
  author={Klatte, Diethard and Kummer, Bernd},
  journal={Math. Program.},
  volume={168},
  pages={673--716},
  year={2018},
  publisher={Springer}
}

@article{bonnans1998sensitivity,
  title={Sensitivity analysis of optimization problems under second order regular constraints},
  author={Bonnans, J Fr{\'e}d{\'e}ric and Cominetti, Roberto and Shapiro, Alexander},
  journal={Math. Oper. Res.},
  volume={23},
  number={4},
  pages={806--831},
  year={1998},
  publisher={INFORMS}
}

@misc{liu2023dual,
  title={Dual {Newton} proximal point algorithm for solution paths of the {$\ell_1$}-regularized logistic regression},
  author={Liu, Yong-Jin and Zhou, Weimi},
  howpublished={arXiv preprint arXiv:2310.19353},
  year={2023}
}

@inproceedings{koh2007method,
  title={A method for large-scale {$\ell_1$}-regularized logistic regression},
  author={Koh, Kwangmoo and Kim, Seung-Jean and Boyd, Stephen},
  booktitle={Proceedings of the Twenty-Second AAAI Conference on Artificial Intelligence},
  pages={565--571},
  year={2007}
}

@article{cumida:2019,
	title        = {{CuMiDa}: An extensively curated microarray database for benchmarking and testing of machine learning approaches in cancer research},
	author       = {Feltes, B.C. and Chandelier, E. B. and Grisci, B. I. and Dorn, M.},
	year         = 2019,
	journal      = {J. Comput. Biol.},
	volume       = 26,
	number       = 4,
	pages        = {376--386}
}

@book{rockafellar1970,
  author    = {Rockafellar, R. Tyrrell},
  title     = {Convex Analysis},
  year      = {1970},
  publisher = {Princeton University Press}
}

@inproceedings{moore2011l1,
  title={{$\ell_1$} and {$\ell_2$} regularization for multiclass hinge loss models},
  author={Moore, Robert and DeNero, John},
  booktitle={MLSLP},
  pages={1--5},
  year={2011}
}

@article{cui2018portfolio,
  title={Portfolio optimization with nonparametric value at risk: A block coordinate descent method},
  author={Cui, Xueting and Sun, Xiaoling and Zhu, Shushang and Jiang, Rujun and Li, Duan},
  journal={INFORMS J. Comput.},
  volume={30},
  number={3},
  pages={454--471},
  year={2018},
  publisher={INFORMS}
}
\end{document}